\documentclass[reqno]{amsart}

\usepackage[utf8]{inputenc}
\usepackage{graphicx}
\usepackage{mathtools}
\usepackage[parfill]{parskip}
\usepackage{tikz}
\usetikzlibrary{arrows, backgrounds, calc, chains, decorations, patterns, positioning, shapes}

\usepackage{amsfonts}
\usepackage{amsmath}
\usepackage{amsrefs}
\usepackage{amssymb}
\usepackage{amstext}
\usepackage{amsthm}

\usepackage{caption}
\usepackage{enumitem}
\usepackage{geometry}
\usepackage{hyperref}
\usepackage{cleveref}
\usepackage{subcaption}

\newcommand{\area}{\mathsf{area}}
\newcommand{\dinv}{\mathsf{dinv}}
\newcommand{\tdinv}{\mathsf{tdinv}}
\newcommand{\cdinv}{\mathsf{cdinv}}
\newcommand{\pdinv}{\mathsf{pdinv}}
\newcommand{\maxtdinv}{\mathsf{maxtdinv}}
\newcommand{\sweep}{\mathsf{sweep}}

\newcommand{\RP}{\mathsf{RP}} 
\newcommand{\RD}{\mathsf{RD}} 
\newcommand{\LRD}{\mathsf{LRD}} 
\newcommand{\LRP}{\mathsf{LRP}} 

\newcommand{\Ht}{\widetilde{H}}

\newcommand{\Z}{\mathbb{Z}}

\newcommand{\<}{\langle}
\renewcommand{\>}{\rangle}

\theoremstyle{plain}

\theoremstyle{definition}
\newtheorem{theorem}{Theorem}[section]
\newtheorem{conjecture}[theorem]{Conjecture}

\newtheorem{definition}[theorem]{Definition}
\newtheorem{lemma}[theorem]{Lemma}

\newtheorem{problem}[theorem]{Problem}

\theoremstyle{remark}

\newtheorem{example}[theorem]{Example}

\makeatletter%
\renewenvironment{proof}[1][\proofname]{%
	\par\pushQED{\qed}\normalfont%
	\topsep6\p@\@plus6\p@\relax
	\trivlist\item[\hskip\labelsep\bfseries#1\@addpunct{.}]%
	\ignorespaces
}{%
	\qedhere 
}
\makeatother

\makeatletter%
\DeclareRobustCommand*{\bfseries}{%
	\not@math@alphabet\bfseries\mathbf
	\fontseries\bfdefault\selectfont
	\boldmath
}
\makeatother

\title[Rectangular paths and Delta conjectures]{Rectangular analogues of the square paths conjecture and the univariate Delta conjecture}
\author{Alessandro Iraci \and Roberto Pagaria \and Giovanni Paolini \and Anna Vanden Wyngaerd}

\address{Alessandro Iraci \newline \textup{Universit\'e du Qu\'ebec \`a Montr\'eal, LACIM}\\ 201 Av. du Pr\'esident-Kennedy, Montr\'eal, QC H2X 3Y7 \\ Canada.}
\email{iraci.alessandro@uqam.ca}

\address{Roberto Pagaria \newline \textup{Università di Bologna, Dipartimento di Matematica}\\ Piazza di Porta San Donato 5 - 40126 Bologna\\ Italy.}
\email{roberto.pagaria@unibo.it}

\address{Giovanni Paolini \newline \textup{California Institute of Technology}\\ Pasadena CA, United States.}
\email{paolini@caltech.edu}

\address{Anna Vanden Wyngaerd  \newline \textup{IRIF, Université de Paris Cité}\\ France.}
\email{avw@irif.fr}

\begin{document}

\begin{abstract}
    In this paper, we extend the rectangular side of the shuffle conjecture by stating a rectangular analogue of the square paths conjecture. In addition, we describe a set of combinatorial objects and one statistic that are a first step towards a rectangular extension of (the rise version of) the Delta conjecture, and of (the rise version of) the Delta square conjecture, corresponding to the case $q=1$ of an expected general statement. We also prove our new rectangular paths conjecture in the special case when the sides of the rectangle are coprime.
\end{abstract}
\maketitle

\section{Introduction}

In the 90's, Garsia and Haiman set out to prove the Schur positivity of the (modified) Macdonald polynomials by showing them to be the bi-graded Frobenius characteristic of certain Garsia-Haiman modules \cite{Garsia-Haiman-PNAS-1993}. Their prediction was confirmed in 2001, when Haiman used the algebraic geometry of the Hilbert sheme to prove that the dimension of their modules equals $n!$ \cite{Haiman-nfactorial-2001}, thus proving the $n!$ theorem. In the course of these developments, it became clear that there were remarkable connections to be found between Macdonald polynomial theory and representation theory of the symmetric group. For example, during their quest for Macdonald positivity, Garsia and Haiman introduced the $\mathfrak{S}_n$-module of \emph{diagonal harmonics}, i.e.\ the coinvariants of the diagonal action of $\mathfrak{S}_n$ on polynomials in two sets of $n$ variables, and they conjectured that its Frobenius characteristic is given by $\nabla e_n$, where $\nabla$ is the \emph{nabla} operator on symmetric functions introduced in \cite{Bergeron-Garsia-Haiman-Tesler-Positivity-1999}, which acts diagonally on Macdonald polynomials. Haiman proved this conjecture in 2002 \cite{Haiman-Vanishing-2002}. 

The combinatorial side of things solidified when Haglund, Haiman, Loehr, Remmel, and Ulyanov then formulated the so called \emph{shuffle conjecture}  \cite{HHLRU-2005}, i.e.\ they predicted a combinatorial formula for $\nabla e_n$ in terms of labelled Dyck paths, which are lattice paths using North and East steps going from $(0,0)$ to $(n,n)$ and staying weakly above the line connecting these two points (called the \emph{main diagonal}). Several years later, Haglund, Morse and Zabrocki conjectured a \emph{compositional} refinement of the shuffle conjecture, which also specified all the points where the Dyck paths returns to main diagonal \cite{Haglund-Morse-Zabrocki-2012}. This was the statement later proved by Carlsson and Mellit in \cite{Carlsson-Mellit-ShuffleConj-2018}, implying the \emph{shuffle theorem}. 

Over the years, this subject has revealed itself to be extremely fruitful and to have striking connections to other fields of mathematics including elliptical Hall algebras, affine Hecke algebras, Springer fibers, the homology of torus knots and the shuffle algebra of symmetric functions. 

In this paper, we add a few (conjectural) formulas to the substantial list of variants and generalisations inspired by the the succes story of the shuffle theorem; that is, equations with a symmetric function related to Macdonald polynomials on one side and lattice paths combinatorics on the other. Furthermore, we support one of these conjectures by proving a non-trivial special case. 

One of the earliest shuffle-like formulas was conjectured in 2007 Loehr and Warrington \cite{Loehr-Warrington-square-2007}. They predicted an expression of $\nabla \omega(p_n)$ in terms of \emph{square paths}, i.e.\ lattice paths from $(0,0)$ to $(n,n)$ using only North and East steps and ending with an East step (without a the restriction of staying above the main diagonal). Their formula was proved by Sergel in \cite{Leven-2016} to be a consequence of the shuffle theorem. 

Next, Haglund, Remmel and Wilson formulated the \emph{Delta conjecture} \cite{Haglund-Remmel-Wilson-2018}, a pair of conjectures for the symmetric function $\Delta'_{e_{n-k-1}}e_n$ in terms of decorated Dyck paths, where $k$ decorations are placed on either \emph{rises} or \emph{valleys} of the path. The symmetric function operator $\Delta'_f$ acts diagonally on the Macdonald polynomials and generalises $\nabla$, in a sense. 
The rise version of the Delta conjecture was proved by D'Adderio and Mellit in \cite{DAdderioMellit2022CompositionalDelta}, using the compositional refinement in \cite{DAdderio-Iraci-VandenWyngaerd-Theta-2021}. In \cite{DAdderio-Iraci-VandenWyngaerd-DeltaSquare-2019}, the authors stated a \emph{Delta square conjecture} (still open at the moment), which extends (the rise version of) the Delta conjecture in the same fashion as the square paths theorem extends the shuffle theorem. The valley version also has similar extensions \cites{Qiu-Wilson-2020, Iraci-VandenWyngaerd-Valley-Square-2021}, but it lacks a compositional version and it is still open.

Around the same time as the formulation of the Delta conjecture, the story has been extended to rectangular Dyck paths: paths from $(0,0)$ to $(m,n)$ staying above the main diagonal. In \cite{Bergeron-Garsia-Sergel-Xin-2016}, building on the work in \cite{Gorsky-Negut-2015}, Bergeron, Garsia, Sergel, and Xin conjectured that a certain symmetric function related to the elliptic Hall algebra studied by Schiffmann and Vasserot \cite{Schiffmann-Vasserot-2011} can be expressed in terms of rectangular Dyck paths. Their prediction was recently proved by Mellit \cite{mellit2021toric}.

In this paper, we state a rectangular analogue of the square paths conjecture, where the combinatorial objects are lattice paths from $(0,0)$ to $(m,n)$ ending with an East step. Our main result is the proof the special case of our conjecture where the sides of the rectangle are coprime. Moreover, using the Theta operators (first introduced in \cite{DAdderio-Iraci-VandenWyngaerd-Theta-2021}), we conjecture the special case $q=1$ of a rectangular analogue of (the rise version of) the Delta conjecture and the Delta square conjecture, in terms of rectangular paths that lie above some horizontal translation of the \emph{broken diagonal}, a ``decorated'' analogue of the diagonal of the rectangle that turns out to be necessary to describe the right set of combinatorial objects.

\section{Symmetric functions}

For all the undefined notations and the unproven identities, we refer to \cite{DAdderioIraciVandenWyngaerd2022TheBible}*{Section~1}, where definitions, proofs and/or references can be found. 

We denote by $\Lambda$ the graded algebra of symmetric functions with coefficients in $\mathbb{Q}(q,t)$, and by $\<\, , \>$ the \emph{Hall scalar product} on $\Lambda$, defined by declaring that the Schur functions form an orthonormal basis.

The standard bases of the symmetric functions that will appear in our calculations are the monomial $\{m_\lambda\}_{\lambda}$, complete $\{h_{\lambda}\}_{\lambda}$, elementary $\{e_{\lambda}\}_{\lambda}$, power $\{p_{\lambda}\}_{\lambda}$ and Schur $\{s_{\lambda}\}_{\lambda}$ bases.

For a partition $\mu \vdash n$, we denote by \[ \Ht_\mu \coloneqq \Ht_\mu[X] = \Ht_\mu[X; q,t] = \sum_{\lambda \vdash n} \widetilde{K}_{\lambda \mu}(q,t) s_{\lambda} \] the \emph{(modified) Macdonald polynomials}, where \[ \widetilde{K}_{\lambda \mu} \coloneqq \widetilde{K}_{\lambda \mu}(q,t) = K_{\lambda \mu}(q,1/t) t^{n(\mu)} \] are the \emph{(modified) Kostka coefficients} (see \cite{Haglund-Book-2008}*{Chapter~2} for more details). 

Macdonald polynomials form a basis of the algebra of symmetric functions $\Lambda$. This is a modification of the basis introduced by Macdonald \cite{Macdonald-Book-1995}.

If we identify the partition $\mu$ with its Ferrer diagram, i.e.\ with the collection of cells $\{(i,j)\mid 1\leq i\leq \mu_j, 1\leq j\leq \ell(\mu)\}$, then for each cell $c\in \mu$ we refer to the \emph{arm}, \emph{leg}, \emph{co-arm} and \emph{co-leg} (denoted respectively by $a_\mu(c), l_\mu(c), a_\mu'(c), l_\mu'(c)$) as the number of cells in $\mu$ that are strictly to the right, below, to the left and above $c$ in $\mu$, respectively (see Figure~\ref{fig:notation}).

\begin{figure}
	\centering
    \begin{tikzpicture}[scale=0.4]
		\draw[gray,opacity=.6](0,0) grid (15,10);
		\fill[white] (1,-0.1)|-(3,1) |- (6,3) |- (9,7) |- (15.1,8) |- (1,-.1);
		\fill[blue, opacity=.15] (0,7) rectangle (9,8) (3,3) rectangle (4,10); 
		\fill[blue, opacity=.5] (3,7) rectangle (4,8);
		\draw (6,7.5) node {\tiny{Arm}} (3.5,5) node[rotate=90] {\tiny{Leg}} (3.5, 9) node[rotate = 90] {\tiny{Co-leg}} (1.5,7.5) node {\tiny{Co-arm}} ;
	\end{tikzpicture}
	\caption{Arm, leg, co-arm, and co-leg of a cell of a partition.}
	\label{fig:notation}
\end{figure}
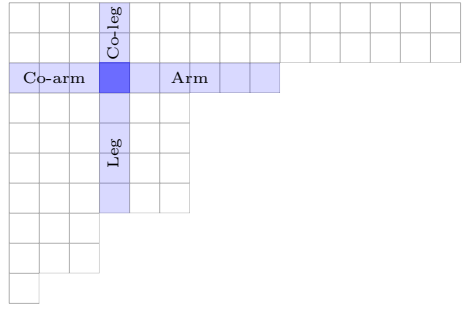

Let $M \coloneqq (1-q)(1-t)$. For every partition $\mu$, we define the following constants:

\[
	B_{\mu} \coloneqq B_{\mu}(q,t) = \sum_{c \in \mu} q^{a_{\mu}'(c)} t^{l_{\mu}'(c)}, \qquad
	\Pi_{\mu} \coloneqq \Pi_{\mu}(q,t) = \prod_{c \in \mu / (1)} (1-q^{a_{\mu}'(c)} t^{l_{\mu}'(c)}). \]

We will make extensive use of the \emph{plethystic notation} (cf. \cite{Haglund-Book-2008}*{Chapter~1}). We also need several linear operators on $\Lambda$.

\begin{definition}[\protect{\cite{Bergeron-Garsia-ScienceFiction-1999}*{3.11}}]
	\label{def:nabla}
	We define the linear operator $\nabla \colon \Lambda \rightarrow \Lambda$ on the eigenbasis of Macdonald polynomials as \[ \nabla \Ht_\mu = e_{\lvert \mu \rvert}[B_\mu] \Ht_\mu. \]
\end{definition}

\begin{definition}
	\label{def:pi}
	We define the linear operator $\mathbf{\Pi} \colon \Lambda \rightarrow \Lambda$ on the eigenbasis of Macdonald polynomials as \[ \mathbf{\Pi} \Ht_\mu = \Pi_\mu \Ht_\mu \] where we conventionally set $\Pi_{\varnothing} \coloneqq 1$.
\end{definition}

\begin{definition}
	\label{def:delta}
	For $f \in \Lambda$, we define the linear operators $\Delta_f, \Delta'_f \colon \Lambda \rightarrow \Lambda$ on the eigenbasis of Macdonald polynomials as \[ \Delta_f \Ht_\mu = f[B_\mu] \Ht_\mu, \qquad \qquad \Delta'_f \Ht_\mu = f[B_\mu-1] \Ht_\mu. \]
\end{definition}

Observe that on the vector space of homogeneous symmetric functions of degree $n$, denoted by $\Lambda^{(n)}$, the operator $\nabla$ equals $\Delta_{e_n}$.

\begin{definition}[\protect{\cite{DAdderio-Iraci-VandenWyngaerd-Theta-2021}*{(28)}}]
	\label{def:theta}
	 For any symmetric function $f \in \Lambda^{(n)}$ we define the \emph{Theta operators} on $\Lambda$ in the following way: for every $F \in \Lambda^{(m)}$ we set
	\begin{equation*}
		\Theta_f F  \coloneqq 
		\left\{\begin{array}{ll}
			0 & \text{if } n \geq 1 \text{ and } m=0 \\
			f \cdot F & \text{if } n=0 \text{ and } m=0 \\
			\mathbf{\Pi} f \left[\frac{X}{M}\right] \mathbf{\Pi}^{-1} F & \text{otherwise}
		\end{array}
		\right. ,
	\end{equation*}
and we extend by linearity the definition to any $f, F \in \Lambda$.
\end{definition}

It is clear that $\Theta_f$ is linear. In addition, if $f$ is homogeneous of degree $k$, then so is $\Theta_f$:
\[\Theta_f \Lambda^{(n)} \subseteq \Lambda^{(n+k)} \qquad \text{ for } f \in \Lambda^{(k)}. \]

Finally, we need to refer to \cite{Bergeron-Garsia-Sergel-Xin-2016}*{Algorithm~4.1} (see also \cite{Bergeron-Garsia-Sergel-Xin-Remarkable-2016}*{Definition~1.1, Theorem~2.5}).

\begin{definition}
    Let $m, n > 0$. Let $a,b,c,d \in \mathbb{N}$ such that $a+c=m$, $b+d=n$, $ad-bc = \gcd(m,n)$. We recursively define $Q_{m,n}$ as an operator on $\Lambda$ by \[ Q_{m,n} = \frac{1}{M} \left( Q_{c,d} Q_{a,b} - Q_{a,b} Q_{c,d} \right), \] with base cases \[ Q_{1,0} = D_0 = \mathsf{id} - M \Delta_{e_1} \quad \text{ and } \quad Q_{0,1} = - \underline{e_1} \] (where $\underline{f}$ is the multiplication by $f$).
\end{definition}

\begin{definition}
    For a coprime pair $(a,b)$ and $f \in \Lambda^{(d)}$, we define $F_{a,b}(f)$ as follows. Let \[ f = \sum_{\lambda \vdash d} c_\lambda(q,t) \left( \frac{qt}{qt-1} \right)^{\ell(\lambda)} h_\lambda \left[ \frac{1-qt}{qt} X \right]. \] Then, we define \[ F_{a,b}(f) \coloneqq \sum_{\lambda \vdash d} c_\lambda(q,t) \prod_{i=1}^{\ell(\lambda)} Q_{\lambda_i a, \lambda_i b}(1). \]
\end{definition}

For our convenience, we use the shorthands \[ e_{m,n} \coloneqq F_{a,b}(e_d), \qquad p_{m,n} \coloneqq F_{a,b}(p_d) \] where $m = ad, n = bd$, and $\gcd(a,b) = 1$. Beware: $e_{4,2} = F_{2,1}(e_2)$, but $e_{42} = e_4 e_2$. 

\section{Combinatorial definitions}

The objects we are concerned with are \emph{rectangular Dyck paths} and \emph{rectangular paths}. All the following definitions are classical for rectangular Dyck paths \cite{Bergeron-Garsia-Sergel-Xin-2016} and new for rectangular paths.

\subsection{Rectangular paths}

\begin{definition}
    A \emph{rectangular path} of size $m \times n$ is a lattice path composed of unit North and East steps, going from $(0,0)$ to $(m,n)$, and ending with an East step. A \emph{rectangular Dyck path} is a rectangular path that lies weakly above the diagonal $my = nx$ (called \emph{main diagonal}).
\end{definition}

\begin{figure}
    \centering
    \begin{tikzpicture}[scale=0.6]
    	\draw[gray!60, thin] (0,0) grid (7,9) (11/9,0) -- (7+11/9,9);
    				
    	\draw[blue!60, line width=2pt] (0,0) -- (0,1) -- (1,1) -- (2,1) -- (2,2) -- (2,3) -- (2,4) -- (3,4) -- (4,4) -- (4,5) -- (4,6) -- (4,7) -- (5,7) -- (6,7) -- (6,8) -- (6,9) -- (7,9);
        \draw[dashed, gray!60] (0,0) -- (7,9);
    \end{tikzpicture}
    \caption{A $7 \times 9$ rectangular path with its base diagonal and the main diagonal (dashed).}\label{fig:rectangular-path}
\end{figure}
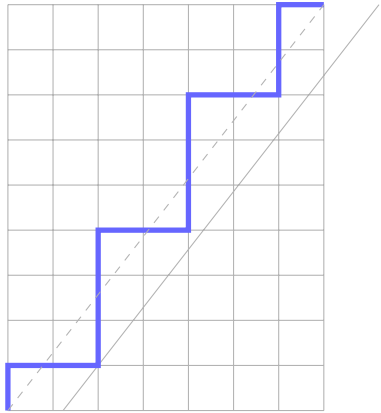

We denote the sets of rectangular paths and rectangular Dyck paths of size $m \times n$ as $\RP(m,n)$ and $\RD(m,n)$, respectively.

\begin{definition}
    For a $m \times n$ rectangular path $\pi$, let $a_i$ be the (signed) horizontal distance between the starting point of the $i$-th North step and the main diagonal. We define the \emph{area word} of the path to be the sequence $(a_1, \dots, a_n)$. Set $s \coloneqq - \min \{a_i \mid 1 \leq i \leq n\}$, which we call the \emph{shift} of the path. Note that $s = 0$ if $\pi$ is a rectangular Dyck path, and $s > 0$ otherwise.
\end{definition}

\begin{definition}
    We call the diagonal $my = n(x-s)$, which is the lowest diagonal that intersects the path, the \emph{base diagonal}.
\end{definition}

\begin{definition}
    The \emph{area} of a rectangular path $\pi$ is $\area(\pi) \coloneqq \sum_{i=1}^{n} \lfloor a_i + s \rfloor$. This is the number of whole squares that lie entirely between the path $\pi$ and its base diagonal.
\end{definition}
For example, the path in Figure~\ref{fig:rectangular-path} has area word
\[\left(0,-\frac{11}{9}, -\frac{4}{9}, \frac{1}{3}, -\frac{8}{9}, -\frac{1}{9}, \frac{2}{3}, -\frac{5}{9}, \frac{2}{9}\right) \approx
(0, \, -1.22,\, -0.44, \, 0.33, \, -0.88, \, -0.11, \, 0.66, \, -0.55, \, 0.22). \]
Thus, its shift is $\frac{11}{9}$ and its area is $5$.

\subsection{Decorated rectangular paths}

In a similar fashion as the rise version of the Delta conjecture \cite{Haglund-Remmel-Wilson-2018} (which is now a theorem \cites{DAdderioMellit2022CompositionalDelta,Blasiak-Haiman-Morse-Pun-Seeling-Extended-Delta-2021}), we introduce the concept of \emph{decorated rises} for rectangular paths.

\begin{definition}
    The \emph{rises} of a rectangular path are the indices of the rows containing a North step that immediately follows another North step. A \emph{decorated rectangular path} is a rectangular path with a given subset $dr$ of its rises.
\end{definition}

\begin{definition}
    For a decorated rectangular path of size $(m+k) \times (n+k)$ with $k$ decorated rises, we define the \emph{broken diagonal} to be the broken segment built as follows. Let $(x_1, y_1) = (0,0)$, then for $1 \leq i < n+k$, define
    \begin{equation*}
		(x_{i+1}, y_{i+1}) =
		\left\{\begin{array}{ll}
			(x_i + \frac{m}{n}, y_i+1) & \text{if } i \not \in dr \\
			(x_i+1, y_i+1) & \text{if } i \in dr. \\
		\end{array}
        \right.
	\end{equation*}
    The broken diagonal is the broken segment joining $(x_i, y_i)$ and $(x_{i+1}, y_{i+1})$ for all $i$, that is, the line that starts at $(0,0)$ and the proceeds with slope $\frac{n}{m}$ in rows not containing decorated rises, and with slope $1$ in rows that contain decorated rises.
\end{definition}

Note that, if the path has no decorated rises, then the broken diagonal coincides with the main diagonal.

\begin{definition}
    We define a \emph{decorated rectangular Dyck path} to be a decorated rectangular path that lies weakly above the broken diagonal.
\end{definition}

See Figure~\ref{fig:decorated-dyck} for an example of such a path. We use a $\ast$ to mark the decorated rises.

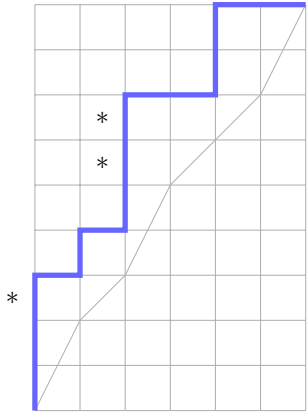
\begin{figure}
    \centering
    \begin{tikzpicture}[scale=.6]
        \draw[step=1.0, gray!60, thin] (0,0) grid (6,9);
    
        \begin{scope}
            \clip (0,0) rectangle (6,9);
            \draw[gray!60, thin] (0,0) -- (0/1, 0) -- (1/2, 1) -- (1/1, 2) -- (2/1, 3) -- (5/2, 4) -- (3/1, 5) -- (4/1, 6) -- (5/1, 7) -- (11/2, 8) -- (6/1, 9);
        \end{scope}
    
        \draw[blue!60, line width=2pt] (0,0) -- (0,1) -- (0,2) -- (0,3) -- (1,3) -- (1,4) -- (2,4) -- (2,5) -- (2,6) -- (2,7) -- (3,7) -- (4,7) -- (4,8) -- (4,9) -- (5,9) -- (6,9);

        \draw (-0.5,2.5) node {$\ast$};
        \draw (1.5,5.5) node {$\ast$};
        \draw (1.5,6.5) node {$\ast$};
    \end{tikzpicture}
    
    \caption{A decorated rectangular Dyck path with its broken diagonal.}
    \label{fig:decorated-dyck}
\end{figure}

The definitions of \emph{area word} and \emph{area} extend to decorated paths as well, using the broken diagonal in place of the main diagonal.

\begin{definition}
    For a $(m+k) \times (n+k)$ decorated rectangular path $(\pi, dr)$ with $k$ decorated rises, let $a_i$ be the horizontal distance between the starting point of the $i$-th North step and the broken diagonal. We define the \emph{area word} of the path as the sequence $a_1, \dots, a_{n+k}$. We define $s \coloneqq - \min \{a_i \mid 1 \leq i \leq n+k\}$ to be the shift of the path.
\end{definition}

\begin{definition}
    We define the \emph{area} of a decorated rectangular path $\pi$ as \[ \area(\pi) \coloneqq \sum_{i \not \in dr} \lfloor a_i + s \rfloor. \]
\end{definition}
The area of the path in Figure~\ref{fig:decorated-dyck} is equal to $3$.

\subsection{Labelled paths}

Finally, we need to introduce labelled objects.

\begin{definition}
    A \emph{labelling} of a (decorated) rectangular (Dyck) path is an assignment of a positive integer label to each North step of the path, such that consecutive North steps are assigned strictly increasing labels. A \emph{labelled (decorated) rectangular (Dyck) path} is a (decorated) rectangular (Dyck) path together with a labelling.
\end{definition}

We say that a labelling is \emph{standard} if the set of labels is $[n] \coloneqq \{1, \dots, n\}$, where $n$ is the height of the path.

\begin{figure}
    \centering
    \begin{minipage}{.45\textwidth}
        \centering
        \begin{tikzpicture}[scale=0.6]
            \draw[gray!60, thin] (0,0) grid (7,9) (11/9,0) -- (7+11/9,9);
    				
            \draw[blue!60, line width = 1.6pt] (0,0) -- (0,1) -- (1,1) -- (2,1) -- (2,2) -- (2,3) -- (2,4) -- (3,4) -- (3,5) -- (3,6) -- (3,7) -- (4,7) -- (5,7) -- (5,8) -- (5,9) -- (6,9) -- (7,9);
            
            \draw
            (0.5,0.5) circle (0.4 cm) node {$1$}
            (2.5,1.5) circle (0.4 cm) node {$1$}
            (2.5,2.5) circle (0.4 cm) node {$4$}
            (2.5,3.5) circle (0.4 cm) node {$7$}
            (3.5,4.5) circle (0.4 cm) node {$2$}
            (3.5,5.5) circle (0.4 cm) node {$4$}
            (3.5,6.5) circle (0.4 cm) node {$7$}
            (5.5,7.5) circle (0.4 cm) node {$1$}
            (5.5,8.5) circle (0.4 cm) node {$2$};
        \end{tikzpicture}
    \end{minipage}%
    \begin{minipage}{.45\textwidth}
        \centering
        \begin{tikzpicture}[scale=.6]
            \draw[step=1.0, gray!60, thin] (0,0) grid (7,9);
        
            \begin{scope}
                \clip (0,0) rectangle (7,9);
                \draw[gray!60, thin] (0,0) -- (0/1, 0) -- (2/3, 1) -- (4/3, 2) -- (7/3, 3) -- (3/1, 4) -- (11/3, 5) -- (14/3, 6) -- (17/3, 7) -- (19/3, 8) -- (7/1, 9);
            \end{scope}
        
            \draw[blue!60, line width=2pt] (0,0) -- (0,1) -- (0,2) -- (0,3) -- (1,3) -- (1,4) -- (2,4) -- (2,5) -- (2,6) -- (2,7) -- (3,7) -- (4,7) -- (4,8) -- (4,9) -- (5,9) -- (6,9) -- (7,9);
        
            \draw (0.5,0.5) circle (0.4cm) node {$1$};
            \draw (0.5,1.5) circle (0.4cm) node {$2$};
            \draw (0.5,2.5) circle (0.4cm) node {$4$};
            \draw (1.5,3.5) circle (0.4cm) node {$7$};
            \draw (2.5,4.5) circle (0.4cm) node {$2$};
            \draw (2.5,5.5) circle (0.4cm) node {$4$};
            \draw (2.5,6.5) circle (0.4cm) node {$7$};
            \draw (4.5,7.5) circle (0.4cm) node {$1$};
            \draw (4.5,8.5) circle (0.4cm) node {$2$};
            \draw (-0.5,2.5) node {$\ast$};
            \draw (1.5,5.5) node {$\ast$};
            \draw (1.5,6.5) node {$\ast$};
        \end{tikzpicture}      
    \end{minipage}
        \caption{A $7 \times 9$ labelled rectangular path (left) and labelled decorated Dyck path (right).}
    \end{figure}
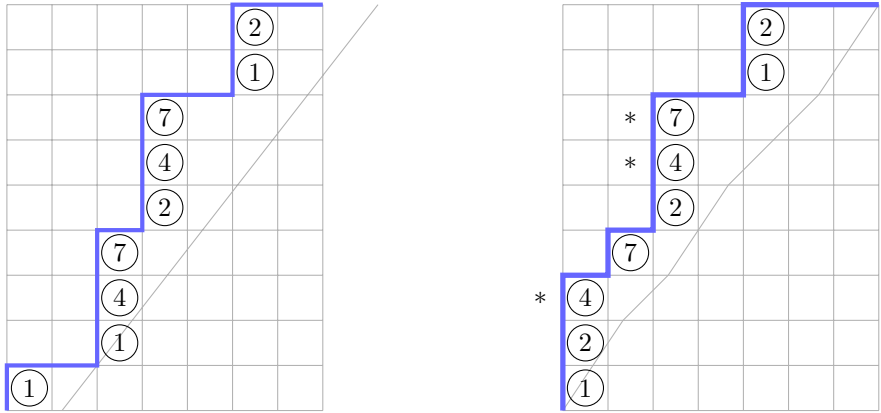

We denote by $w_i$ the label assigned to the $i$-th North step of the path. 

We also denote the sets of labelled rectangular paths and labelled rectangular Dyck paths of size $m \times n$ as $\LRP(m,n)$ and $\LRD(m,n)$ respectively, and the sets of labelled decorated rectangular paths and labelled decorated rectangular Dyck paths of size $(m+k) \times (n+k)$ with $k$ decorated rises as $\LRP(m+k,n+k)^{\ast k}$ and $\LRD(m+k,n+k)^{\ast k}$, respectively.

\begin{definition}
    Given a labelled (decorated) rectangular (Dyck) path $(\pi, dr, w)$, we define $x^w = \prod_i x_{w_i}$. With an abuse of notation, we will sometimes write $\pi$ to mean $(\pi, dr, w)$, in which case we will have $x^\pi = x^w$.
\end{definition}

Given a rectangular (Dyck) path $\pi$, the cells in the rectangular grid going from $(0,0)$ to $(m,n)$ that lie above the path form the Ferrer's diagram of a partition $\mu(\pi)$.

Here we extend the definition of dinv given in \cite{Bergeron-Garsia-Sergel-Xin-2016} (see also \cite{mellit2021toric}) for rectangular Dyck paths to any rectangular path. We will describe it in two different ways.

\begin{definition}
    Let $\pi$ be a $m \times n$ rectangular path, and let $1 \leq i, j \leq n$. We say that $i$ \emph{attacks} $j$ in $\pi$ (or $(i,j)$ is an \emph{attack relation} for $\pi$) if \[ (a_i, i) <_{\text{lex}} (a_j, j) <_{\text{lex}} (a_i + {\textstyle\frac{m}{n}}, i). \]
\end{definition}

At this point, we can define the dinv of an unlabelled path.

\begin{definition}
    We define the \emph{path dinv} of a rectangular path $\pi$ as \[ \pdinv(\pi) \coloneqq \# \left\{ c \in \mu(\pi) \mid \textstyle\frac{a}{\ell+1} \leq \frac{m}{n} < \frac{a+1}{\ell} \right\} \] where $a = a_\mu(c)$ and $\ell = \ell_\mu(c)$, and the second inequality always holds if $\ell = 0$.
\end{definition}

For labelled paths, we need some extra steps.

\begin{definition}
    We define the \emph{temporary dinv} of a labelled rectangular path $(\pi, w)$ as 
    \[ \tdinv(\pi) \coloneqq \# \{ 1 \leq i, j \leq n \mid w_i < w_j \text{ and } i \text{ attacks } j \}. \]
\end{definition}

\begin{definition}
    We define the \emph{maximal temporary dinv} of a rectangular path $\pi$ as 
    \[ \maxtdinv(\pi) \coloneqq \# \{ 1 \leq i, j \leq n \mid i \text{ attacks } j \}. \] Note that this is the same as $\max\{\tdinv(\pi, w) \mid w \in W(\pi)\}$, where $W(\pi)$ is the set of all possible labellings of $\pi$.
\end{definition}

The following is a simpler description for the difference $\pdinv(\pi) - \maxtdinv(\pi)$, given in \cite{Hicks-Sergel-2015}.

\begin{definition}
    We define the \emph{dinv correction} of a rectangular path $\pi$ as \[ \cdinv(\pi) \coloneqq \# \left\{ c \in \mu(\pi) \mid \textstyle\frac{a+1}{\ell+1} \leq \frac{m}{n} < \frac{a}{\ell} \right\} - \# \left\{ c \in \mu(\pi) \mid \textstyle\frac{a}{\ell} \leq \frac{m}{n} < \frac{a+1}{\ell+1} \right\}, \] where $a = a_\mu(c)$ and $\ell = \ell_\mu(c)$.
\end{definition}

We will provide a visual interpretation for the $\tdinv$ and $\cdinv$ later in the section.

\begin{theorem}[\cite{Hicks-Sergel-2015}*{Theorem~2}]
    \label{thm:cdinv-dyck}
    For any rectangular Dyck path $\pi$, we have \[ \cdinv(\pi) = \pdinv(\pi) - \maxtdinv(\pi). \]
\end{theorem}

We extend this result to all rectangular paths, without the restriction of lying above the main diagonal.

\begin{theorem}
    \label{thm:cdinv}
    For any rectangular path $\pi$, we have \[ \cdinv(\pi) = \pdinv(\pi) - \maxtdinv(\pi) - \# \{ i \mid a_i(\pi) < 0 \} - \# \left\{ i \mid a_i(\pi) < -\frac{m}{n} \right\}. \]
\end{theorem}

\begin{proof}
    Let $\pi'$ be the path obtained from $\pi$ by adding $n$ North steps at the beginning, and $m$ East steps at the end. By construction, $\mu(\pi') = \mu(\pi)$ and the slope is the same, so $\cdinv(\pi') = \cdinv(\pi)$. By \Cref{thm:cdinv-dyck}, this quantity is also equal to $\pdinv(\pi') - \maxtdinv(\pi')$. But again, $\pdinv(\pi)$ only depends on $\mu(\pi)$, so $\pdinv(\pi') = \pdinv(\pi)$.

    We only need to compare $\maxtdinv(\pi)$ and $\maxtdinv(\pi')$. It is immediate that $(i,j)$ is an attack relation in $\pi$ if and only if $(n+i, n+j)$ is an attack relation in $\pi'$, so we only need to count attack relations in $\pi'$ where either $i \leq n$ or $j \leq n$. Since the first $n$ steps of $\pi'$ are all North steps by construction, we cannot possibly have attack relations where both $i$ and $j$ are at most $n$.

    We have that, whenever $a_i(\pi) < 0$ (i.e.\ the corresponding North step begins strictly below the main diagonal), $n+i$ is attacked exactly once in $\pi'$ by some $j \leq n$. In fact, we have $0 \leq a_{n+i}(\pi') = m + a_i(\pi) < m$, and since $a_j(\pi') = \frac{m}{n}(j-1)$ for $j \leq n$, there exists exactly one $j$ such that $\frac{m}{n}(j-1) \leq a_{n+i}(\pi') < \frac{m}{n}j$ (which is exactly the attack relation, as $j < n+i$).

    For the same reason, whenever $a_i(\pi) < -\frac{m}{n}$ (i.e.\ the corresponding North step ends strictly below the main diagonal), $n+i$ attacks exactly one $j \leq n$ in $\pi'$. In fact, if that is the case, we have $a_{n+i}(\pi') = m + a_i(\pi) \leq \frac{m}{n} (n-1)$, so there exists exactly one $j$ such that $a_{n+i}(\pi') < \frac{m}{n}(j-1) \leq a_{n+i}(\pi') + \frac{m}{n}j$ (which is exactly the attack relation, as $n+i > j$).

    Summarising, we have
    \begin{align*}
        \cdinv(\pi) & = \cdinv(\pi') \\
        & = \pdinv(\pi') - \maxtdinv(\pi') \\
        & = \pdinv(\pi) - \maxtdinv(\pi') \\
        & = \pdinv(\pi) - \maxtdinv(\pi) - \# \{ i \mid a_i(\pi) < 0 \} - \# \left\{ i \mid a_i(\pi) < -\frac{m}{n} \right\}
    \end{align*}
    as desired.
\end{proof}

Note that the term $\# \{ i \mid a_i(\pi) < 0 \}$ counts the number of North steps of the path that begin below the main diagonal, in the same fashion as in the tertiary dinv (or bonus dinv) for square paths \cites{Loehr-Warrington-square-2007, Leven-2016}.
To obtain a unified definition of dinv of rectangular paths that matches the expected symmetric functions, it turns out that we have to keep that term and disregard the term $\# \{ i \mid a_i(\pi) < -\frac{m}{n} \}$. This finally leads us to the following definition.

\begin{definition}
    We define the \emph{dinv} of a labelled rectangular path $(\pi, w)$ as \[ \dinv(\pi, w) \coloneqq \tdinv(\pi, w) + \cdinv(\pi) +  \# \{ i \mid a_i(\pi) < 0 \}. \]
\end{definition}

\begin{figure*}[p]
    \centering
\begin{minipage}{.3\textwidth}
    \begin{tikzpicture}[scale=.8]
        
    \draw[step=1.0, gray!60, thin] (0,0) grid (5,7);
    
    \begin{scope}
        \clip (0,0) rectangle (5,7);
        \draw[gray!60, thin](0,0) -- (5,7);
        \draw[gray!60, thin](0,1) -- (5,7+1);
    \end{scope}
    \draw[blue!60, line width=2pt] (0,0) -- (0,1) -- (0,2) -- (1,2) -- (1,3) -- (2,3) -- (3,3) -- (3,4) -- (3,5) -- (4,5) -- (4,6) -- (4,7) -- (5,7);
    \draw[red, line width=2pt] (0,0) -- (0,1);
    \filldraw[red] (4,6) circle(2pt);
    
    \draw (0.5,0.5) circle (0.4cm) node {$2$};
    \draw (0.5,1.5) circle (0.4cm) node {$3$};
    \draw (1.5,2.5) circle (0.4cm) node {$1$};
    \draw (3.5,3.5) circle (0.4cm) node {$2$};
    \draw (3.5,4.5) circle (0.4cm) node {$4$};
    \draw (4.5,5.5) circle (0.4cm) node {$3$};
    \draw (4.5,6.5) circle (0.4cm) node {$4$};
\end{tikzpicture}
\end{minipage}%
\begin{minipage}{.3\textwidth}
    \begin{tikzpicture}[scale=.8]
        
        \draw[step=1.0, gray!60, thin] (0,0) grid (5,7);

        \begin{scope}
        \clip (0,0) rectangle (5,7);
        \draw[gray!60, thin](0,2) -- (5,7+2);
        \draw[gray!60, thin](0,1) -- (5,7+1);
    \end{scope}
    
    \draw[blue!60, line width=2pt] (0,0) -- (0,1) -- (0,2) -- (1,2) -- (1,3) -- (2,3) -- (3,3) -- (3,4) -- (3,5) -- (4,5) -- (4,6) -- (4,7) -- (5,7);
    
    \draw[red, line width=2pt] (0,1) -- (0,2);
    
    \draw (0.5,0.5) circle (0.4cm) node {$2$};
    \draw (0.5,1.5) circle (0.4cm) node {$3$};
    \draw (1.5,2.5) circle (0.4cm) node {$1$};
    \draw (3.5,3.5) circle (0.4cm) node {$2$};
    \draw (3.5,4.5) circle (0.4cm) node {$4$};
    \draw (4.5,5.5) circle (0.4cm) node {$3$};
    \draw (4.5,6.5) circle (0.4cm) node {$4$};
\end{tikzpicture}
\end{minipage}%
\begin{minipage}{.3\textwidth}    
    \begin{tikzpicture}[scale=.8]
        \draw[step=1.0, gray!60, thin] (0,0) grid (5,7);
        
        \begin{scope}
            \clip (0,0) rectangle (5,7);
            \draw[gray!60, thin](-5+1,-7+2) -- (5+1,7+2);
            \draw[gray!60, thin](-5+1,-7+3) -- (5+1,7+3);
        \end{scope}
        
        \fill[gray,opacity = .5] (0,2) rectangle (1,3);
        \draw[blue!60, line width=2pt] (0,0) -- (0,1) -- (0,2) -- (1,2) -- (1,3) -- (2,3) -- (3,3) -- (3,4) -- (3,5) -- (4,5) -- (4,6) -- (4,7) -- (5,7);

        \draw[red, line width=2pt] (1,2) -- (1,3);
        \filldraw[red] (0,1) circle(2pt);
        
        \draw (0.5,0.5) circle (0.4cm) node {$2$};
        \draw (0.5,1.5) circle (0.4cm) node {$3$};
        \draw (1.5,2.5) circle (0.4cm) node {$1$};
        \draw (3.5,3.5) circle (0.4cm) node {$2$};
        \draw (3.5,4.5) circle (0.4cm) node {$4$};
        \draw (4.5,5.5) circle (0.4cm) node {$3$};
        \draw (4.5,6.5) circle (0.4cm) node {$4$};
    \end{tikzpicture}
\end{minipage}

\bigskip

\begin{minipage}{.3\textwidth}
    \begin{tikzpicture}[scale=.8]
        \draw[step=1.0, gray!60, thin] (0,0) grid (5,7);

        \begin{scope}
            \clip (0,0) rectangle (5,7);
            \draw[gray!60, thin](-5+3,-7+3) -- (5+3,7+3);
            \draw[gray!60, thin](-5+3,-7+4) -- (5+3,7+4);
        \end{scope}
        
        \fill[gray,opacity = .5] (2,3) rectangle (3,4);
        
        \draw[blue!60, line width=2pt] (0,0) -- (0,1) -- (0,2) -- (1,2) -- (1,3) -- (2,3) -- (3,3) -- (3,4) -- (3,5) -- (4,5) -- (4,6) -- (4,7) -- (5,7);
        
        \draw[red, line width=2pt] (3,3) -- (3,4);
        \filldraw[red] (4,5) circle(2pt);

        \draw (0.5,0.5) circle (0.4cm) node {$2$};
        \draw (0.5,1.5) circle (0.4cm) node {$3$};
        \draw (1.5,2.5) circle (0.4cm) node {$1$};
        \draw (3.5,3.5) circle (0.4cm) node {$2$};
        \draw (3.5,4.5) circle (0.4cm) node {$4$};
        \draw (4.5,5.5) circle (0.4cm) node {$3$};
        \draw (4.5,6.5) circle (0.4cm) node {$4$};
    \end{tikzpicture}
\end{minipage}%
\begin{minipage}{.3\textwidth}
    \begin{tikzpicture}[scale=.8]
        \draw[step=1.0, gray!60, thin] (0,0) grid (5,7);

        \begin{scope}
            \clip (0,0) rectangle (5,7);
            \draw[gray!60, thin](-5+3,-7+5) -- (5+3,7+5);
            \draw[gray!60, thin](-5+3,-7+4) -- (5+3,7+4);
        \end{scope}

        \draw[blue!60, line width=2pt] (0,0) -- (0,1) -- (0,2) -- (1,2) -- (1,3) -- (2,3) -- (3,3) -- (3,4) -- (3,5) -- (4,5) -- (4,6) -- (4,7) -- (5,7);

        \draw[red, line width=2pt] (3,4) -- (3,5);

        \draw (0.5,0.5) circle (0.4cm) node {$2$};
        \draw (0.5,1.5) circle (0.4cm) node {$3$};
        \draw (1.5,2.5) circle (0.4cm) node {$1$};
        \draw (3.5,3.5) circle (0.4cm) node {$2$};
        \draw (3.5,4.5) circle (0.4cm) node {$4$};
        \draw (4.5,5.5) circle (0.4cm) node {$3$};
        \draw (4.5,6.5) circle (0.4cm) node {$4$};
    \end{tikzpicture}
\end{minipage}%
\begin{minipage}{.3\textwidth}
    \begin{tikzpicture}[scale=.8]
        \draw[step=1.0, gray!60, thin] (0,0) grid (5,7);

        \begin{scope}
            \clip (0,0) rectangle (5,7);
            \draw[gray!60, thin](-5+4,-7+5) -- (5+4,7+5);
            \draw[gray!60, thin](-5+4,-7+6) -- (5+4,7+6);
        \end{scope}

        \fill[gray,opacity = .5] (3,5) rectangle (4,6);
        
        \draw[blue!60, line width=2pt] (0,0) -- (0,1) -- (0,2) -- (1,2) -- (1,3) -- (2,3) -- (3,3) -- (3,4) -- (3,5) -- (4,5) -- (4,6) -- (4,7) -- (5,7);

        \draw[red, line width=2pt] (4,5) -- (4,6);
        \filldraw[red] (3,4) circle(2pt);

        \draw (0.5,0.5) circle (0.4cm) node {$2$};
        \draw (0.5,1.5) circle (0.4cm) node {$3$};
        \draw (1.5,2.5) circle (0.4cm) node {$1$};
        \draw (3.5,3.5) circle (0.4cm) node {$2$};
        \draw (3.5,4.5) circle (0.4cm) node {$4$};
        \draw (4.5,5.5) circle (0.4cm) node {$3$};
        \draw (4.5,6.5) circle (0.4cm) node {$4$};
    \end{tikzpicture}
\end{minipage}

\bigskip

\begin{minipage}{.3\textwidth}
    \begin{tikzpicture}[scale=.8]
        \draw[step=1.0, gray!60, thin] (0,0) grid (5,7);

        \begin{scope}
            \clip (0,0) rectangle (5,7);

            \draw[gray!60, thin](-5+4,-7+7) -- (5+4,7+7);
            \draw[gray!60, thin](-5+4,-7+6) -- (5+4,7+6);
        \end{scope}

        \fill[gray,opacity = .5] (1,6) rectangle (2,7);

        \draw[blue!60, line width=2pt] (0,0) -- (0,1) -- (0,2) -- (1,2) -- (1,3) -- (2,3) -- (3,3) -- (3,4) -- (3,5) -- (4,5) -- (4,6) -- (4,7) -- (5,7);

        \draw[red, line width=2pt] (4,6) -- (4,7);

        \draw (0.5,0.5) circle (0.4cm) node {$2$};
        \draw (0.5,1.5) circle (0.4cm) node {$3$};
        \draw (1.5,2.5) circle (0.4cm) node {$1$};
        \draw (3.5,3.5) circle (0.4cm) node {$2$};
        \draw (3.5,4.5) circle (0.4cm) node {$4$};
        \draw (4.5,5.5) circle (0.4cm) node {$3$};
        \draw (4.5,6.5) circle (0.4cm) node {$4$};
    \end{tikzpicture}
\end{minipage}
\caption{Calculation of the dinv of a rectangular square path.}
\label{fig:dinv}
\end{figure*}
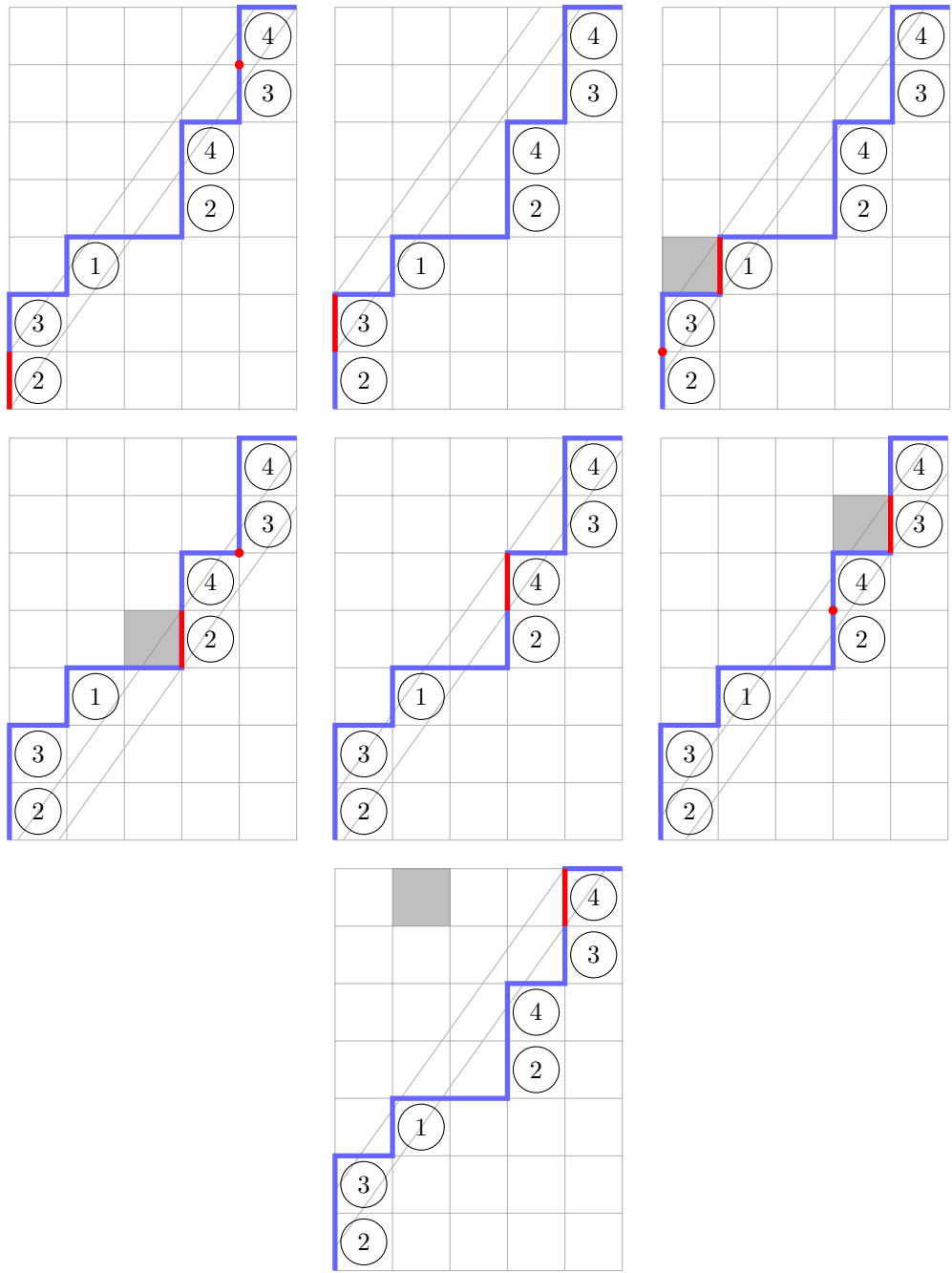

We now give a visual interpretation of the various summands. 

The temporary dinv counts all pairs of North steps $(i,j)$ such that $w_i < w_j$ and the $j$-th North step begins between the line $y = \frac{n}{m}(x + a_i)$ and the line $y =\frac{n}{m}(x + a_i) + 1$, with ties broken by comparing $i$ and $j$. In Figure~\ref{fig:dinv}, we have drawn these two lines for all North steps of the path and marked the beginnings of North steps contained between them and that satisfy the condition on the label. We see that the contribution to the dinv is $4$.

The dinv correction is split into two parts. The first summand counts the number of cells $c$ above the path such that the two lines parallel to the main diagonal and starting from the endpoints of the East step below $c$ both intersect the North step to the right of $c$ (bottom endpoint excluded, but top endpoint included). 
The second summand counts the number of cells $c$ above the path such that the two lines parallel to the main diagonal and starting from the endpoints of the North step to the right of $c$ both intersect the East step below $c$ (right endpoint included, but left endpoint excluded). Notice that the two sets cannot simultaneously be non-empty, the first one being empty if $m \leq n$ and the second one being empty if $m \geq n$. In Figure~\ref{fig:dinv}, we have a path of size $5 \times 7$  so the second term is $0$. We have greyed out the cells counted in the third term, giving a contribution to the dinv of $-4$.

The bonus dinv, as previously mentioned, counts the number of North steps of the path that begin below the main diagonal. In Figure~\ref{fig:dinv} there are $3$ North steps starting below the main diagonal.

Thus the path in Figure~\ref{fig:dinv} has dinv equal to $3$.

\section{Conjectures}

With the previous definitions in mind, we can state the \emph{rectangular shuffle theorem} \cite{mellit2021toric} and several new conjectures, which were verified by computer for all paths with semiperimeter $m+n$ up to $13$.

\begin{theorem}{\cite{mellit2021toric}}
    For any $m, n \in \mathbb{N}$, we have \[ e_{m,n} = \sum_{\pi \in \LRD(m,n)} q^{\dinv(\pi)} t^{\area(\pi)} x^\pi. \]
    \label{thm:rectangular-dyck}
\end{theorem}

Conjecturally, we extend this result to rectangular paths, as follows.

\begin{conjecture}
    \label{conjecture:rectangular-paths}
    For any $m, n \in \mathbb{N}$, and $d = \gcd(m,n)$, we have \[ \frac{[m]_q}{[d]_q} p_{m,n} = \sum_{\pi \in \LRP(m,n)} q^{\dinv(\pi)} t^{\area(\pi)} x^\pi. \]
\end{conjecture}

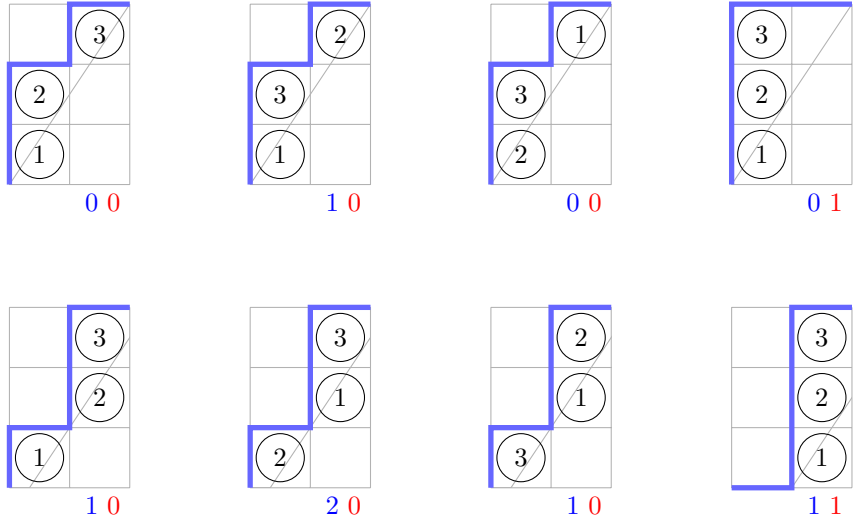
\begin{figure}
\centering
\begin{tikzpicture}[scale=0.8]
    \draw[draw=none, use as bounding box] (-1,-1) rectangle (3,4);
    \draw[step=1.0, gray!60, thin] (0,0) grid (2,3);

    \begin{scope}
        \clip (0,0) rectangle (2,3);
        \draw[gray!60, thin] (0,0) -- (0/1, 0) -- (2/3, 1) -- (4/3, 2) -- (2/1, 3);
    \end{scope}

    \draw[blue!60, line width=2pt] (0,0) -- (0,1) -- (0,2) -- (1,2) -- (1,3) -- (2,3);

    \draw (0.5,0.5) circle (0.4cm) node {$1$};
    \draw (0.5,1.5) circle (0.4cm) node {$2$};
    \draw (1.5,2.5) circle (0.4cm) node {$3$};

      \node[below left] at (2,0) { \color{blue}{$0$} \color{red}{$0$}};
\end{tikzpicture}%
\begin{tikzpicture}[scale=0.8]
    \draw[draw=none, use as bounding box] (-1,-1) rectangle (3,4);
    \draw[step=1.0, gray!60, thin] (0,0) grid (2,3);

    \begin{scope}
        \clip (0,0) rectangle (2,3);
        \draw[gray!60, thin] (0,0) -- (0/1, 0) -- (2/3, 1) -- (4/3, 2) -- (2/1, 3);
    \end{scope}

    \draw[blue!60, line width=2pt] (0,0) -- (0,1) -- (0,2) -- (1,2) -- (1,3) -- (2,3);

    \draw (0.5,0.5) circle (0.4cm) node {$1$};
    \draw (0.5,1.5) circle (0.4cm) node {$3$};
    \draw (1.5,2.5) circle (0.4cm) node {$2$};

      \node[below left] at (2,0) { \color{blue}{$1$} \color{red}{$0$}};
\end{tikzpicture}%
\begin{tikzpicture}[scale=0.8]
    \draw[draw=none, use as bounding box] (-1,-1) rectangle (3,4);
    \draw[step=1.0, gray!60, thin] (0,0) grid (2,3);

    \begin{scope}
        \clip (0,0) rectangle (2,3);
        \draw[gray!60, thin] (0,0) -- (0/1, 0) -- (2/3, 1) -- (4/3, 2) -- (2/1, 3);
    \end{scope}

    \draw[blue!60, line width=2pt] (0,0) -- (0,1) -- (0,2) -- (1,2) -- (1,3) -- (2,3);

    \draw (0.5,0.5) circle (0.4cm) node {$2$};
    \draw (0.5,1.5) circle (0.4cm) node {$3$};
    \draw (1.5,2.5) circle (0.4cm) node {$1$};

      \node[below left] at (2,0) { \color{blue}{$0$} \color{red}{$0$}};
\end{tikzpicture}%
\begin{tikzpicture}[scale=0.8]
    \draw[draw=none, use as bounding box] (-1,-1) rectangle (3,4);
    \draw[step=1.0, gray!60, thin] (0,0) grid (2,3);

    \begin{scope}
        \clip (0,0) rectangle (2,3);
        \draw[gray!60, thin] (0,0) -- (0/1, 0) -- (2/3, 1) -- (4/3, 2) -- (2/1, 3);
    \end{scope}

    \draw[blue!60, line width=2pt] (0,0) -- (0,1) -- (0,2) -- (0,3) -- (1,3) -- (2,3);

    \draw (0.5,0.5) circle (0.4cm) node {$1$};
    \draw (0.5,1.5) circle (0.4cm) node {$2$};
    \draw (0.5,2.5) circle (0.4cm) node {$3$};

      \node[below left] at (2,0) { \color{blue}{$0$} \color{red}{$1$}};
\end{tikzpicture}

\begin{tikzpicture}[scale=0.8]
    \draw[draw=none, use as bounding box] (-1,-1) rectangle (3,4);
    \draw[step=1.0, gray!60, thin] (0,0) grid (2,3);

    \begin{scope}
        \clip (0,0) rectangle (2,3);
        \draw[gray!60, thin] (0,0) -- (1/3, 0) -- (1/1, 1) -- (5/3, 2) -- (7/3, 3);
    \end{scope}

    \draw[blue!60, line width=2pt] (0,0) -- (0,1) -- (1,1) -- (1,2) -- (1,3) -- (2,3);

    \draw (0.5,0.5) circle (0.4cm) node {$1$};
    \draw (1.5,1.5) circle (0.4cm) node {$2$};
    \draw (1.5,2.5) circle (0.4cm) node {$3$};

      \node[below left] at (2,0) { \color{blue}{$1$} \color{red}{$0$}};
\end{tikzpicture}%
\begin{tikzpicture}[scale=0.8]
    \draw[draw=none, use as bounding box] (-1,-1) rectangle (3,4);
    \draw[step=1.0, gray!60, thin] (0,0) grid (2,3);

    \begin{scope}
        \clip (0,0) rectangle (2,3);
        \draw[gray!60, thin] (0,0) -- (1/3, 0) -- (1/1, 1) -- (5/3, 2) -- (7/3, 3);
    \end{scope}

    \draw[blue!60, line width=2pt] (0,0) -- (0,1) -- (1,1) -- (1,2) -- (1,3) -- (2,3);

    \draw (0.5,0.5) circle (0.4cm) node {$2$};
    \draw (1.5,1.5) circle (0.4cm) node {$1$};
    \draw (1.5,2.5) circle (0.4cm) node {$3$};

      \node[below left] at (2,0) { \color{blue}{$2$} \color{red}{$0$}};
\end{tikzpicture}%
\begin{tikzpicture}[scale=0.8]
    \draw[draw=none, use as bounding box] (-1,-1) rectangle (3,4);
    \draw[step=1.0, gray!60, thin] (0,0) grid (2,3);

    \begin{scope}
        \clip (0,0) rectangle (2,3);
        \draw[gray!60, thin] (0,0) -- (1/3, 0) -- (1/1, 1) -- (5/3, 2) -- (7/3, 3);
    \end{scope}

    \draw[blue!60, line width=2pt] (0,0) -- (0,1) -- (1,1) -- (1,2) -- (1,3) -- (2,3);

    \draw (0.5,0.5) circle (0.4cm) node {$3$};
    \draw (1.5,1.5) circle (0.4cm) node {$1$};
    \draw (1.5,2.5) circle (0.4cm) node {$2$};

      \node[below left] at (2,0) { \color{blue}{$1$} \color{red}{$0$}};
\end{tikzpicture}%
\begin{tikzpicture}[scale=0.8]
    \draw[draw=none, use as bounding box] (-1,-1) rectangle (3,4);
    \draw[step=1.0, gray!60, thin] (0,0) grid (2,3);

    \begin{scope}
        \clip (0,0) rectangle (2,3);
        \draw[gray!60, thin] (0,0) -- (1/1, 0) -- (5/3, 1) -- (7/3, 2) -- (3/1, 3);
    \end{scope}

    \draw[blue!60, line width=2pt] (0,0) -- (1,0) -- (1,1) -- (1,2) -- (1,3) -- (2,3);

    \draw (1.5,0.5) circle (0.4cm) node {$1$};
    \draw (1.5,1.5) circle (0.4cm) node {$2$};
    \draw (1.5,2.5) circle (0.4cm) node {$3$};

      \node[below left] at (2,0) { \color{blue}{$1$} \color{red}{$1$}};
\end{tikzpicture}
\caption{The set of $2 \times 3$ standard rectangular paths, with their dinv (in blue) and area (in red).}
\label{fig:rectangular-paths-23}
\end{figure}

\begin{example}
    Let $m=2$ and $n=3$. In \Cref{conjecture:rectangular-paths}, we can check for example that the Hilbert series (that is, the scalar product with $h_{1^n}$) coincides with the sum over all $2 \times 3$ standard rectangular paths of the monomial $q^{\dinv(\pi)} t^{\area(\pi)}$. In fact, we have
    \[ \frac{[2]_q}{[1]_q} \langle p_{2,3}, h_{1^3} \rangle = (1+q)(q+t+2) = 1 + q + 1 + t + q + q^2 + q + qt, \]
    which coincides with the values in Figure~\ref{fig:rectangular-paths-23}.
\end{example}

We also have (univariate) analogues of the Delta conjecture and the Delta square conjecture for rectangular (Dyck) paths, using Theta operators.

\begin{conjecture}
    \label{conj:rectangular-dyck-delta}
    For any $m, n \in \mathbb{N}$, we have \[ \left. \Theta_{e_k} e_{m,n} \right\rvert_{q=1} = \sum_{\pi \in \LRD(m+k,n+k)^{\ast k}} t^{\area(\pi)} x^\pi. \]
\end{conjecture}

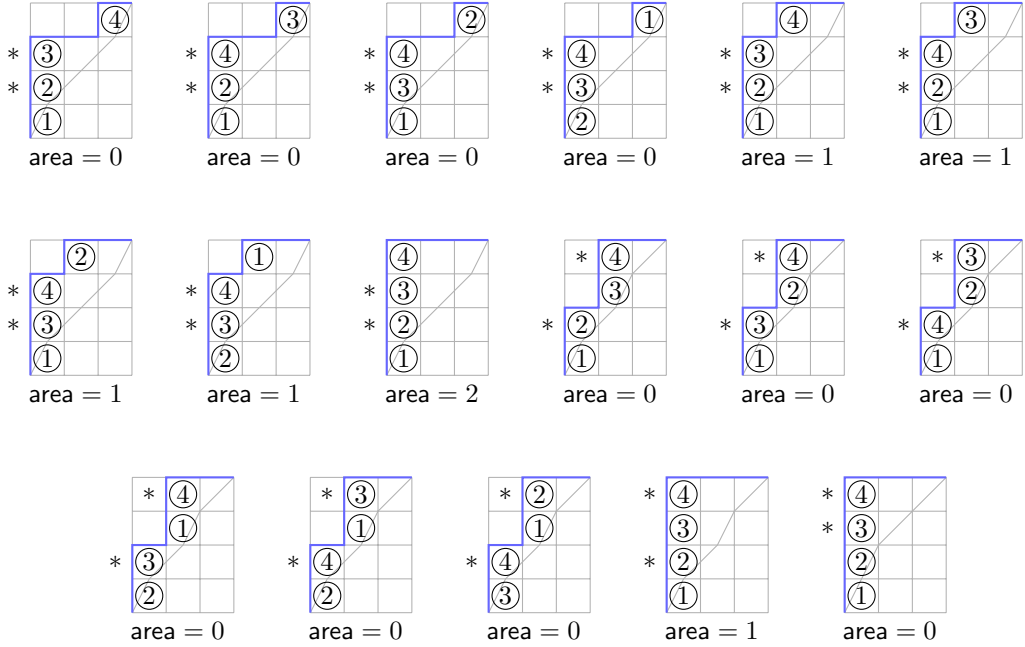
\begin{figure}
\begin{tikzpicture}[scale=0.45]
    \draw[draw=none, use as bounding box] (-1,-1) rectangle (4,5);
    \draw[step=1.0, gray!60, thin] (0,0) grid (3,4);

    \begin{scope}
        \clip (0,0) rectangle (3,4);
        \draw[gray!60, thin] (0,0) -- (0/1, 0) -- (1/2, 1) -- (3/2, 2) -- (5/2, 3) -- (3/1, 4);
    \end{scope}

    \draw[blue!60,thick] (0,0) -- (0,1) -- (0,2) -- (0,3) -- (1,3) -- (2,3) -- (2,4) -- (3,4);

    \draw (0.5,0.5) circle (0.4cm) node {$1$};
    \draw (0.5,1.5) circle (0.4cm) node {$2$};
    \draw (0.5,2.5) circle (0.4cm) node {$3$};
    \draw (2.5,3.5) circle (0.4cm) node {$4$};
    \draw (-0.5,1.5) node {$\ast$};
    \draw (-0.5,2.5) node {$\ast$};

      \node[below left] at (3,0) {{$\area = 0$}};
\end{tikzpicture}
\begin{tikzpicture}[scale=0.45]
    \draw[draw=none, use as bounding box] (-1,-1) rectangle (4,5);
    \draw[step=1.0, gray!60, thin] (0,0) grid (3,4);

    \begin{scope}
        \clip (0,0) rectangle (3,4);
        \draw[gray!60, thin] (0,0) -- (0/1, 0) -- (1/2, 1) -- (3/2, 2) -- (5/2, 3) -- (3/1, 4);
    \end{scope}

    \draw[blue!60,thick] (0,0) -- (0,1) -- (0,2) -- (0,3) -- (1,3) -- (2,3) -- (2,4) -- (3,4);

    \draw (0.5,0.5) circle (0.4cm) node {$1$};
    \draw (0.5,1.5) circle (0.4cm) node {$2$};
    \draw (0.5,2.5) circle (0.4cm) node {$4$};
    \draw (2.5,3.5) circle (0.4cm) node {$3$};
    \draw (-0.5,1.5) node {$\ast$};
    \draw (-0.5,2.5) node {$\ast$};

      \node[below left] at (3,0) {{$\area = 0$}};
\end{tikzpicture}
\begin{tikzpicture}[scale=0.45]
    \draw[draw=none, use as bounding box] (-1,-1) rectangle (4,5);
    \draw[step=1.0, gray!60, thin] (0,0) grid (3,4);

    \begin{scope}
        \clip (0,0) rectangle (3,4);
        \draw[gray!60, thin] (0,0) -- (0/1, 0) -- (1/2, 1) -- (3/2, 2) -- (5/2, 3) -- (3/1, 4);
    \end{scope}

    \draw[blue!60,thick] (0,0) -- (0,1) -- (0,2) -- (0,3) -- (1,3) -- (2,3) -- (2,4) -- (3,4);

    \draw (0.5,0.5) circle (0.4cm) node {$1$};
    \draw (0.5,1.5) circle (0.4cm) node {$3$};
    \draw (0.5,2.5) circle (0.4cm) node {$4$};
    \draw (2.5,3.5) circle (0.4cm) node {$2$};
    \draw (-0.5,1.5) node {$\ast$};
    \draw (-0.5,2.5) node {$\ast$};

      \node[below left] at (3,0) {{$\area = 0$}};
\end{tikzpicture}
\begin{tikzpicture}[scale=0.45]
    \draw[draw=none, use as bounding box] (-1,-1) rectangle (4,5);
    \draw[step=1.0, gray!60, thin] (0,0) grid (3,4);

    \begin{scope}
        \clip (0,0) rectangle (3,4);
        \draw[gray!60, thin] (0,0) -- (0/1, 0) -- (1/2, 1) -- (3/2, 2) -- (5/2, 3) -- (3/1, 4);
    \end{scope}

    \draw[blue!60,thick] (0,0) -- (0,1) -- (0,2) -- (0,3) -- (1,3) -- (2,3) -- (2,4) -- (3,4);

    \draw (0.5,0.5) circle (0.4cm) node {$2$};
    \draw (0.5,1.5) circle (0.4cm) node {$3$};
    \draw (0.5,2.5) circle (0.4cm) node {$4$};
    \draw (2.5,3.5) circle (0.4cm) node {$1$};
    \draw (-0.5,1.5) node {$\ast$};
    \draw (-0.5,2.5) node {$\ast$};

      \node[below left] at (3,0) {{$\area = 0$}};
\end{tikzpicture}
\begin{tikzpicture}[scale=0.45]
    \draw[draw=none, use as bounding box] (-1,-1) rectangle (4,5);
    \draw[step=1.0, gray!60, thin] (0,0) grid (3,4);

    \begin{scope}
        \clip (0,0) rectangle (3,4);
        \draw[gray!60, thin] (0,0) -- (0/1, 0) -- (1/2, 1) -- (3/2, 2) -- (5/2, 3) -- (3/1, 4);
    \end{scope}

    \draw[blue!60,thick] (0,0) -- (0,1) -- (0,2) -- (0,3) -- (1,3) -- (1,4) -- (2,4) -- (3,4);

    \draw (0.5,0.5) circle (0.4cm) node {$1$};
    \draw (0.5,1.5) circle (0.4cm) node {$2$};
    \draw (0.5,2.5) circle (0.4cm) node {$3$};
    \draw (1.5,3.5) circle (0.4cm) node {$4$};
    \draw (-0.5,1.5) node {$\ast$};
    \draw (-0.5,2.5) node {$\ast$};

      \node[below left] at (3,0) {{$\area = 1$}};
\end{tikzpicture}
\begin{tikzpicture}[scale=0.45]
    \draw[draw=none, use as bounding box] (-1,-1) rectangle (4,5);
    \draw[step=1.0, gray!60, thin] (0,0) grid (3,4);

    \begin{scope}
        \clip (0,0) rectangle (3,4);
        \draw[gray!60, thin] (0,0) -- (0/1, 0) -- (1/2, 1) -- (3/2, 2) -- (5/2, 3) -- (3/1, 4);
    \end{scope}

    \draw[blue!60,thick] (0,0) -- (0,1) -- (0,2) -- (0,3) -- (1,3) -- (1,4) -- (2,4) -- (3,4);

    \draw (0.5,0.5) circle (0.4cm) node {$1$};
    \draw (0.5,1.5) circle (0.4cm) node {$2$};
    \draw (0.5,2.5) circle (0.4cm) node {$4$};
    \draw (1.5,3.5) circle (0.4cm) node {$3$};
    \draw (-0.5,1.5) node {$\ast$};
    \draw (-0.5,2.5) node {$\ast$};

      \node[below left] at (3,0) {{$\area = 1$}};
\end{tikzpicture}\newline

\begin{tikzpicture}[scale=0.45]
    \draw[draw=none, use as bounding box] (-1,-1) rectangle (4,5);
    \draw[step=1.0, gray!60, thin] (0,0) grid (3,4);

    \begin{scope}
        \clip (0,0) rectangle (3,4);
        \draw[gray!60, thin] (0,0) -- (0/1, 0) -- (1/2, 1) -- (3/2, 2) -- (5/2, 3) -- (3/1, 4);
    \end{scope}

    \draw[blue!60,thick] (0,0) -- (0,1) -- (0,2) -- (0,3) -- (1,3) -- (1,4) -- (2,4) -- (3,4);

    \draw (0.5,0.5) circle (0.4cm) node {$1$};
    \draw (0.5,1.5) circle (0.4cm) node {$3$};
    \draw (0.5,2.5) circle (0.4cm) node {$4$};
    \draw (1.5,3.5) circle (0.4cm) node {$2$};
    \draw (-0.5,1.5) node {$\ast$};
    \draw (-0.5,2.5) node {$\ast$};

      \node[below left] at (3,0) {{$\area = 1$}};
\end{tikzpicture}
\begin{tikzpicture}[scale=0.45]
    \draw[draw=none, use as bounding box] (-1,-1) rectangle (4,5);
    \draw[step=1.0, gray!60, thin] (0,0) grid (3,4);

    \begin{scope}
        \clip (0,0) rectangle (3,4);
        \draw[gray!60, thin] (0,0) -- (0/1, 0) -- (1/2, 1) -- (3/2, 2) -- (5/2, 3) -- (3/1, 4);
    \end{scope}

    \draw[blue!60,thick] (0,0) -- (0,1) -- (0,2) -- (0,3) -- (1,3) -- (1,4) -- (2,4) -- (3,4);

    \draw (0.5,0.5) circle (0.4cm) node {$2$};
    \draw (0.5,1.5) circle (0.4cm) node {$3$};
    \draw (0.5,2.5) circle (0.4cm) node {$4$};
    \draw (1.5,3.5) circle (0.4cm) node {$1$};
    \draw (-0.5,1.5) node {$\ast$};
    \draw (-0.5,2.5) node {$\ast$};

      \node[below left] at (3,0) {{$\area = 1$}};
\end{tikzpicture}
\begin{tikzpicture}[scale=0.45]
    \draw[draw=none, use as bounding box] (-1,-1) rectangle (4,5);
    \draw[step=1.0, gray!60, thin] (0,0) grid (3,4);

    \begin{scope}
        \clip (0,0) rectangle (3,4);
        \draw[gray!60, thin] (0,0) -- (0/1, 0) -- (1/2, 1) -- (3/2, 2) -- (5/2, 3) -- (3/1, 4);
    \end{scope}

    \draw[blue!60,thick] (0,0) -- (0,1) -- (0,2) -- (0,3) -- (0,4) -- (1,4) -- (2,4) -- (3,4);

    \draw (0.5,0.5) circle (0.4cm) node {$1$};
    \draw (0.5,1.5) circle (0.4cm) node {$2$};
    \draw (0.5,2.5) circle (0.4cm) node {$3$};
    \draw (0.5,3.5) circle (0.4cm) node {$4$};
    \draw (-0.5,1.5) node {$\ast$};
    \draw (-0.5,2.5) node {$\ast$};

      \node[below left] at (3,0) {{$\area = 2$}};
\end{tikzpicture}
\begin{tikzpicture}[scale=0.45]
    \draw[draw=none, use as bounding box] (-1,-1) rectangle (4,5);
    \draw[step=1.0, gray!60, thin] (0,0) grid (3,4);

    \begin{scope}
        \clip (0,0) rectangle (3,4);
        \draw[gray!60, thin] (0,0) -- (0/1, 0) -- (1/2, 1) -- (3/2, 2) -- (2/1, 3) -- (3/1, 4);
    \end{scope}

    \draw[blue!60,thick] (0,0) -- (0,1) -- (0,2) -- (1,2) -- (1,3) -- (1,4) -- (2,4) -- (3,4);

    \draw (0.5,0.5) circle (0.4cm) node {$1$};
    \draw (0.5,1.5) circle (0.4cm) node {$2$};
    \draw (1.5,2.5) circle (0.4cm) node {$3$};
    \draw (1.5,3.5) circle (0.4cm) node {$4$};
    \draw (-0.5,1.5) node {$\ast$};
    \draw (0.5,3.5) node {$\ast$};

      \node[below left] at (3,0) {{$\area = 0$}};
\end{tikzpicture}
\begin{tikzpicture}[scale=0.45]
    \draw[draw=none, use as bounding box] (-1,-1) rectangle (4,5);
    \draw[step=1.0, gray!60, thin] (0,0) grid (3,4);

    \begin{scope}
        \clip (0,0) rectangle (3,4);
        \draw[gray!60, thin] (0,0) -- (0/1, 0) -- (1/2, 1) -- (3/2, 2) -- (2/1, 3) -- (3/1, 4);
    \end{scope}

    \draw[blue!60,thick] (0,0) -- (0,1) -- (0,2) -- (1,2) -- (1,3) -- (1,4) -- (2,4) -- (3,4);

    \draw (0.5,0.5) circle (0.4cm) node {$1$};
    \draw (0.5,1.5) circle (0.4cm) node {$3$};
    \draw (1.5,2.5) circle (0.4cm) node {$2$};
    \draw (1.5,3.5) circle (0.4cm) node {$4$};
    \draw (-0.5,1.5) node {$\ast$};
    \draw (0.5,3.5) node {$\ast$};

      \node[below left] at (3,0) {{$\area = 0$}};
\end{tikzpicture}
\begin{tikzpicture}[scale=0.45]
    \draw[draw=none, use as bounding box] (-1,-1) rectangle (4,5);
    \draw[step=1.0, gray!60, thin] (0,0) grid (3,4);

    \begin{scope}
        \clip (0,0) rectangle (3,4);
        \draw[gray!60, thin] (0,0) -- (0/1, 0) -- (1/2, 1) -- (3/2, 2) -- (2/1, 3) -- (3/1, 4);
    \end{scope}

    \draw[blue!60,thick] (0,0) -- (0,1) -- (0,2) -- (1,2) -- (1,3) -- (1,4) -- (2,4) -- (3,4);

    \draw (0.5,0.5) circle (0.4cm) node {$1$};
    \draw (0.5,1.5) circle (0.4cm) node {$4$};
    \draw (1.5,2.5) circle (0.4cm) node {$2$};
    \draw (1.5,3.5) circle (0.4cm) node {$3$};
    \draw (-0.5,1.5) node {$\ast$};
    \draw (0.5,3.5) node {$\ast$};

      \node[below left] at (3,0) {{$\area = 0$}};
\end{tikzpicture}\newline

\begin{tikzpicture}[scale=0.45]
    \draw[draw=none, use as bounding box] (-1,-1) rectangle (4,5);
    \draw[step=1.0, gray!60, thin] (0,0) grid (3,4);

    \begin{scope}
        \clip (0,0) rectangle (3,4);
        \draw[gray!60, thin] (0,0) -- (0/1, 0) -- (1/2, 1) -- (3/2, 2) -- (2/1, 3) -- (3/1, 4);
    \end{scope}

    \draw[blue!60,thick] (0,0) -- (0,1) -- (0,2) -- (1,2) -- (1,3) -- (1,4) -- (2,4) -- (3,4);

    \draw (0.5,0.5) circle (0.4cm) node {$2$};
    \draw (0.5,1.5) circle (0.4cm) node {$3$};
    \draw (1.5,2.5) circle (0.4cm) node {$1$};
    \draw (1.5,3.5) circle (0.4cm) node {$4$};
    \draw (-0.5,1.5) node {$\ast$};
    \draw (0.5,3.5) node {$\ast$};

      \node[below left] at (3,0) {{$\area = 0$}};
\end{tikzpicture}
\begin{tikzpicture}[scale=0.45]
    \draw[draw=none, use as bounding box] (-1,-1) rectangle (4,5);
    \draw[step=1.0, gray!60, thin] (0,0) grid (3,4);

    \begin{scope}
        \clip (0,0) rectangle (3,4);
        \draw[gray!60, thin] (0,0) -- (0/1, 0) -- (1/2, 1) -- (3/2, 2) -- (2/1, 3) -- (3/1, 4);
    \end{scope}

    \draw[blue!60,thick] (0,0) -- (0,1) -- (0,2) -- (1,2) -- (1,3) -- (1,4) -- (2,4) -- (3,4);

    \draw (0.5,0.5) circle (0.4cm) node {$2$};
    \draw (0.5,1.5) circle (0.4cm) node {$4$};
    \draw (1.5,2.5) circle (0.4cm) node {$1$};
    \draw (1.5,3.5) circle (0.4cm) node {$3$};
    \draw (-0.5,1.5) node {$\ast$};
    \draw (0.5,3.5) node {$\ast$};

      \node[below left] at (3,0) {{$\area = 0$}};
\end{tikzpicture}
\begin{tikzpicture}[scale=0.45]
    \draw[draw=none, use as bounding box] (-1,-1) rectangle (4,5);
    \draw[step=1.0, gray!60, thin] (0,0) grid (3,4);

    \begin{scope}
        \clip (0,0) rectangle (3,4);
        \draw[gray!60, thin] (0,0) -- (0/1, 0) -- (1/2, 1) -- (3/2, 2) -- (2/1, 3) -- (3/1, 4);
    \end{scope}

    \draw[blue!60,thick] (0,0) -- (0,1) -- (0,2) -- (1,2) -- (1,3) -- (1,4) -- (2,4) -- (3,4);

    \draw (0.5,0.5) circle (0.4cm) node {$3$};
    \draw (0.5,1.5) circle (0.4cm) node {$4$};
    \draw (1.5,2.5) circle (0.4cm) node {$1$};
    \draw (1.5,3.5) circle (0.4cm) node {$2$};
    \draw (-0.5,1.5) node {$\ast$};
    \draw (0.5,3.5) node {$\ast$};

      \node[below left] at (3,0) {{$\area = 0$}};
\end{tikzpicture}
\begin{tikzpicture}[scale=0.45]
    \draw[draw=none, use as bounding box] (-1,-1) rectangle (4,5);
    \draw[step=1.0, gray!60, thin] (0,0) grid (3,4);

    \begin{scope}
        \clip (0,0) rectangle (3,4);
        \draw[gray!60, thin] (0,0) -- (0/1, 0) -- (1/2, 1) -- (3/2, 2) -- (2/1, 3) -- (3/1, 4);
    \end{scope}

    \draw[blue!60,thick] (0,0) -- (0,1) -- (0,2) -- (0,3) -- (0,4) -- (1,4) -- (2,4) -- (3,4);

    \draw (0.5,0.5) circle (0.4cm) node {$1$};
    \draw (0.5,1.5) circle (0.4cm) node {$2$};
    \draw (0.5,2.5) circle (0.4cm) node {$3$};
    \draw (0.5,3.5) circle (0.4cm) node {$4$};
    \draw (-0.5,1.5) node {$\ast$};
    \draw (-0.5,3.5) node {$\ast$};

      \node[below left] at (3,0) {{$\area = 1$}};
\end{tikzpicture}
\begin{tikzpicture}[scale=0.45]
    \draw[draw=none, use as bounding box] (-1,-1) rectangle (4,5);
    \draw[step=1.0, gray!60, thin] (0,0) grid (3,4);

    \begin{scope}
        \clip (0,0) rectangle (3,4);
        \draw[gray!60, thin] (0,0) -- (0/1, 0) -- (1/2, 1) -- (1/1, 2) -- (2/1, 3) -- (3/1, 4);
    \end{scope}

    \draw[blue!60,thick] (0,0) -- (0,1) -- (0,2) -- (0,3) -- (0,4) -- (1,4) -- (2,4) -- (3,4);

    \draw (0.5,0.5) circle (0.4cm) node {$1$};
    \draw (0.5,1.5) circle (0.4cm) node {$2$};
    \draw (0.5,2.5) circle (0.4cm) node {$3$};
    \draw (0.5,3.5) circle (0.4cm) node {$4$};
    \draw (-0.5,2.5) node {$\ast$};
    \draw (-0.5,3.5) node {$\ast$};

      \node[below left] at (3,0) {{$\area = 0$}};
\end{tikzpicture}
\caption{The set of $3 \times 4$ standard rectangular Dyck paths with two decorated rises, with their area.}
\label{fig:rectangular-delta}
\end{figure}

\begin{conjecture}
    \label{conj:rectangular-delta}
    For any $m, n \in \mathbb{N}$, and $d = \gcd(m,n)$, we have \[ \left. \frac{[m+k]_q}{[d]_q} \Theta_{e_k} p_{m,n} \right\rvert_{q=1} = \sum_{\pi \in \LRP(m+k,n+k)^{\ast k}} t^{\area(\pi)} x^\pi. \]
\end{conjecture}

Conjectures~\ref{conj:rectangular-dyck-delta} and \ref{conj:rectangular-delta} have been checked by computer up to semiperimeter $13$. See Figure~\ref{fig:rectangular-delta} for the case $m=1$, $n=2$, $k=2$: indeed \[ \left. \< \Theta_{e_2} e_{2,1}, h_{1^4} \> \right\rvert_{q=1} = t^2 + 5t + 11, \] which coincides with the combinatorial expression.

These conjectures bring with themselves a natural open problem.

\begin{problem}
    Find a statistic $\mathsf{qstat} \colon \LRP(m+k,n+k)^{\ast k} \rightarrow \mathbb{N}$ such that \[ \Theta_{e_k} e_{m,n} = \sum_{\pi \in \LRD(m+k,n+k)^{\ast k}} q^{\mathsf{qstat}(\pi)} t^{\area(\pi)} x^\pi \] and \[ \frac{[m+k]_q}{[d]_q} \Theta_{e_k} p_{m,n} = \sum_{\pi \in \LRP(m+k,n+k)^{\ast k}} q^{\mathsf{qstat}(\pi)}t^{\area(\pi)} x^\pi. \]
\end{problem}

Unlike in the square case, simply ignoring the decorations on the rises to compute the dinv does not give the expected $\mathsf{qstat}$.

\section{The sweep process}

In this section, we show that the sweep process in \cite{mellit2021toric}*{Subsection~4.1} also gives the correct outcome for rectangular paths, without the restriction of staying above the main diagonal.

We refer to \cite{mellit2021toric}*{Proposition~3.3} for the definitions of the operators $d_+$ and $d_-$, to \cite{mellit2021toric}*{Subsection~3.5} for the definition of characteristic function of a Dyck path with a marking, and to \cite{mellit2021toric}*{Section~4} and the first paragraph of \cite{mellit2021toric}*{Theorem~4.2} for how they relate to the following sweep process. We do not report all the definitions here because we are only interested in certain combinatorial properties of the sweep process and how they change between rectangular Dyck paths and rectangular paths, rather than in the process itself, but we encourage the interested reader to compare \Cref{thm:sweep} and \cite{mellit2021toric}*{Theorem~4.2}.

\begin{definition}[Sweep process]
    For $\pi \in \RP(m, n)$, define $\sweep(\pi)$ through the algorithm that follows. 
    Initialize $\varphi = 1 \in V_0$.
    Consider a line $l$ with slope $\frac{n}{m} - \epsilon$, with $\epsilon < \frac{1}{(2mn)^2}$ (so that it ``breaks ties'' but does not change the order in which the lattice points are hit with respect to a line with slope $\frac{m}{n}$), which stays fully above $\pi$.
    Move $l$ downward and modify $\varphi$ every time $l$ passes through a lattice point $p$ weakly below $\pi$ and different from $(m, n)$.
    At each lattice point $p$, modify $\varphi$ as follows:
    \begin{itemize}
        \item[(A)] if $p$ is between a NE pair of steps, apply $d_+$;
        \item[(B)] if $p$ is between an EN pair of steps, or $p = (0, 0)$ and the path starts with a N step, apply $d_-$;
        \item[(C)] if $p$ is between a NN pair of steps, apply $q^{-a}\frac{d_-d_+ - d_+d_-}{q-1}$, where $a$ is the number of vertical steps of $\pi$ crossed by $l$ to the right of $p$;
        \item[(D)] if $p$ is between an EE pair of steps, or $p = (0, 0)$ and the path starts with an E step, multiply by $q^a$ (where $a$ is defined as in the previous case);
        \item[(E)] if $p$ is strictly below $\pi$, multiply by $t$.
    \end{itemize}
    The algorithm stops when $l$ is entirely below the path $\pi$.
\end{definition}
See Figure~\ref{fig:sweeping} for an illustration of the sweeping process.

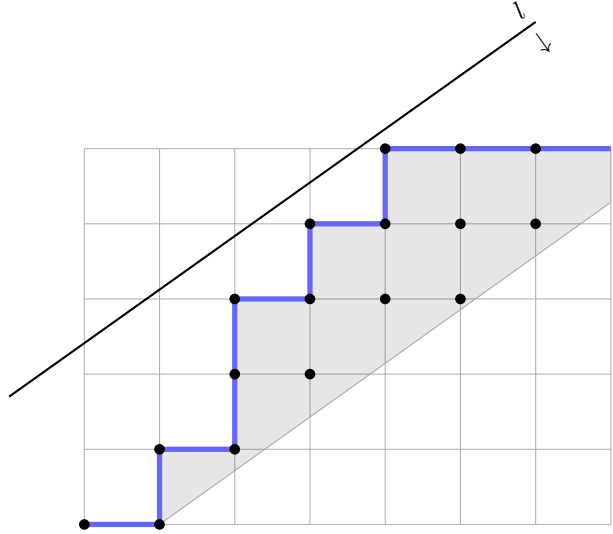
\begin{figure}
    \centering
    \begin{tikzpicture}[scale=1]
        \draw[step=1.0, gray!60, thin] (0,0) grid (7,5);
        \draw[gray!60, thin] (1,0) -- (7,30/7);
        \fill[opacity = .1] (0,0) -- (1,0) -- (1,1) -- (2,1) -- (2,2) -- (2,3) -- (3,3) -- (3,4) -- (4,4) -- (4,5) -- (5,5) -- (6,5) -- (7,5) -- (7,30/7) -- (1,0);
        
        \draw[blue!60, line width=2pt] (0,0) -- (1,0) -- (1,1) -- (2,1) -- (2,2) -- (2,3) -- (3,3) -- (3,4) -- (4,4) -- (4,5) -- (5,5) -- (6,5) -- (7,5);
        \fill 
        (0,0) circle (2pt)
        (1,0) circle (2pt)
        (1,1) circle (2pt)
        (2,1) circle (2pt)
        (2,2) circle (2pt)
        (3,2) circle (2pt)
        (2,3) circle (2pt)
        (3,3) circle (2pt)
        (3,4) circle (2pt)
        (4,3) circle (2pt)
        (4,4) circle (2pt)
        (4,5) circle (2pt)
        (5,3) circle (2pt)
        (5,4) circle (2pt)
        (5,5) circle (2pt)
        (6,4) circle (2pt)
        (6,5) circle (2pt); 
    
        \draw[thick] (-1,1.7) -- (6,6.7 - 0.014) node[pos = .99, above, sloped] {$l$} node[pos = .99, below, sloped] {$\downarrow$};
    \end{tikzpicture}
    \caption{The sweeping process.}\label{fig:sweeping}
\end{figure}

\begin{theorem}
    \label{thm:sweep}
    For $\pi$ any rectangular path, we have \[ \sweep(\pi) = t^{\area(\pi)} \sum_{w \in W(\pi)} q^{\dinv(\pi, w)} x^w, \]
    where $W(\pi)$ is the set of possible labellings of $\pi$.
\end{theorem}

\begin{proof}
    As in \cite{mellit2021toric}*{Theorem~4.2}, plotting the attack relations gives a Dyck path $\tilde\pi$ with a set of marked corners $\Sigma_\pi$ such that
    \[ \chi(\tilde\pi, \Sigma_\pi) = \sum_{w \in W(\pi)} q^{\tdinv(\pi, w)} x^w, \]
    where $\chi(\tilde\pi, \Sigma_\pi)$ is the characteristic function of a Dyck path (see \cite{mellit2021toric}*{Subsection~3.5}) and such that $\chi(\tilde\pi, \Sigma_\pi)$ is the result of the operations $(A), (B)$, and $(C)$ without the factor $q^{-a}$.

    It is also clear that operation $(E)$ gives $t^{\area(\pi)}$, so all that is left to show is that the power of $q$ produced by rules $(C)$ and $(D)$ equals \[ + \# \left\{ c \in \mu(\pi) \mid \textstyle\frac{a+1}{\ell+1} \leq \frac{m}{n} < \frac{a}{\ell} \right\} - \# \left\{ c \in \mu(\pi) \mid \textstyle\frac{a}{\ell} \leq \frac{m}{n} < \frac{a+1}{\ell+1} \right\} + \# \{ i \mid a_i < 0 \}. \]

    Let us again define $\pi'$ to be the path obtained from $\pi$ by adding $n$ North steps at the beginning, and $m$ East steps at the end. Since $\pi'$ is a rectangular Dyck path, by the proof of \cite{mellit2021toric}*{Theorem~4.2} we know that the power of $q$ produced by rules $(C)$ and $(D)$ equals $\cdinv(\pi')$, which is also equal to $\cdinv(\pi)$ as it only depends on $\mu(\pi') = \mu(\pi)$.
    
    We need to compare the power of $q$ produced by rules $(C)$ and $(D)$ applied to $\pi'$ and $\pi$.
    The result is the same for lattice points in between steps of $\pi'$ that were already in $\pi$, as we are not adding any North step to their right. For the lattice points in between the last $m$ East steps of $\pi'$, the exponent of $q$ is always $0$, as their corresponding value of $a$ is $0$.
    
    For the lattice points in between first $n$ steps of $\pi'$, we have to apply rule $(C)$, so their total contribution is equal to minus the number of North steps of $\pi'$ intersected by any line with slope $\frac{n}{m} - \varepsilon$ starting from $(0,j)$ for some $j < n$ which is exactly the number of North steps of $\pi$ finishing strictly below the main diagonal, that is, the number of $i$ such that $a_i(\pi) < - \frac{m}{n}$.
    
    Finally, the point $(0,n)$ in $\pi$ switches from rule $(D)$ to rule $(A)$, or from rule $(B)$ to rule $(C)$, depending whether $\pi$ starts with an East or a North step respectively; in either case, the difference between its contributions in $\pi'$ and in $\pi$ is given by minus the number of North steps of $\pi$ that are crossed by the line with slope $\frac{n}{m} - \varepsilon$ starting from $(0,0)$, which is exactly the number of $i$ such that $- \frac{m}{n} \leq a_i(\pi) < 0$.
    
    In total, we get that the difference in the exponents of $q$ produced by rules $(C)$ and $(D)$ applied to $\pi'$ and $\pi$ is $- \# \{ i \mid a_i(\pi) < 0 \}$, so we have
    \begin{align*}
        \sweep(\pi) & = t^{\area(\pi)} q^{\cdinv(\pi)} q^{\# \{ i \mid a_i(\pi) < 0 \}} \sum_{w \in W(\pi)} q^{\tdinv(\pi, w)} x^w
        & = t^{\area(\pi)} \sum_{w \in W(\pi)} q^{\dinv(\pi, w)} x^w
    \end{align*}
    as desired.    
\end{proof}

\section{The coprime case}

In this section, we prove \Cref{conjecture:rectangular-paths} in the coprime case:

\begin{theorem}
    If $\gcd(m, n) = 1$, then
    \[ [m]_q \, p_{m,n} = \sum_{\pi \in \LRP(m,n)} q^{\dinv(\pi)} t^{\area(\pi)} x^\pi. \]
    \label{thm:coprime-case}
\end{theorem}

\begin{proof}
    Since $\gcd(m,n) = 1$, we have $e_{m,n} = p_{m,n} = F_{m,n}(e_1)$.
    Therefore, in order to prove \Cref{thm:coprime-case}, it is enough to show that the set of (unlabelled) rectangular paths $\RP(m,n)$ can be partitioned into subsets $\mathcal{P}_1, \dotsc, \mathcal{P}_h$ of cardinality $m$ such that:
    
    \begin{enumerate}[label=(\arabic*)]
        \item each $\mathcal{P}_i$ contains exactly one Dyck path $\pi_0 \in \RD(m,n)$;
        \item for each $\mathcal{P}_i$ and $0 \leq k < m$, there exists a (unique) element $\pi_k \in \mathcal{P}_i$ such that $\sweep(\pi_k) = q^k \sweep(\pi_0)$.
    \end{enumerate}
    
    Indeed, if such a partition exists, then
    \begin{align*}
        [m]_q \, p_{m,n} &= [m]_q \, e_{m,n} = [m]_q \sum_{\pi \in \LRD(m,n)} q^{\dinv(\pi)} t^{\area(\pi)} x^\pi \\
        &= [m]_q \sum_{\pi \in \RD(m,n)} \sweep(\pi) \\
        &= \sum_{\pi \in \RP(m,n)} \sweep(\pi) \\
        &= \sum_{\pi \in \LRP(m,n)} q^{\dinv(\pi)} t^{\area(\pi)} x^\pi,
    \end{align*}
    where we used \Cref{thm:rectangular-dyck} in the first line, \Cref{thm:sweep} in the second line, the partition $\RP(m,n) = \mathcal{P}_1 \sqcup \dotsb \sqcup \mathcal{P}_h$ in the third line, and \Cref{thm:sweep} again in the fourth line.
        
    Next, we construct a partition of $\RP(m,n)$ with the desired properties.
    Consider an (unlabelled) rectangular path $\pi \in \RP(m, n)$.
    Denote by $d_i \in \mathbb{Q}$ the signed horizontal distance between the endpoint of the $i$-th horizontal step of $\pi$ and the main diagonal (for $0 \leq i < k$).
    Fix now an integer $k$ with $0 \leq k < m$.
    The $k$-th horizontal step divides the path $\pi$ into two parts $\pi_0$ and $\pi_1$, where $\pi_1$ starts immediately after the $k$-th horizontal step and $\pi_0$ ends with the $k$-th horizontal step.
    Define the path $\phi(\pi)=\phi_k(\pi)$
    as the concatenation of $\pi_1$ followed by $\pi_0$ (we fix $\phi_0 = \mathsf{id}$).
    Also, let
    \[ r(\pi) = r_k(\pi)
    =\begin{cases} \# \left\{ i \mid d_k> d_i \geq 0 \right\} & \textnormal{if } d_k\geq 0 \\
        - \# \left\{ i \mid 0 \geq d_i > d_k \right\} & \textnormal{if } d_k < 0, \end{cases}\]
    that is, up to a sign, the number of horizontal steps whose endpoint lies between the main diagonal and the diagonal parallel to it that passes through the endpoint of the $k$-th horizontal step.
        
    We partition $\RP(m, n)$ as follows.
    If $\pi \in \RD(m,n)$ is the $i$-th Dyck path, define $\mathcal{P}_i = \{ \phi_k(\pi) \mid 0 \leq k < m \}$.
    The sets $\mathcal{P}_1, \dotsc, \mathcal{P}_h$ form a partition of $\RP(m, n)$.
    Since $\gcd(m, n) = 1$, $\mathcal{P}_i$ contains no Dyck path other than $\pi$, so the partition satisfies property (1) above.
    By definition of $r_k$, we have that $\{ r_k(\pi) \mid 0 \leq k < m \} = \{0, 1, \dotsc, m-1\}$. See Figure~\ref{fig:partition RP} for an example.
    Then the partition satisfies property (2) thanks to \Cref{lemma} below.
\end{proof}

\begin{figure}
    \centering
    \begin{minipage}{.99\textwidth}
        \centering
        \begin{tikzpicture}[scale=.6]
            \draw[draw=none, use as bounding box]  (-.5,-.5)rectangle(5.5,6.5);
            \draw[step=1.0, gray!60, thin] (0,0) grid (5,6);
            
            \begin{scope}
                \clip (0,0) rectangle (5,6);
                \draw[gray!60, thin] (0,0) -- (0/1, 0) -- (5/6, 1) -- (5/3, 2) -- (5/2, 3) -- (10/3, 4) -- (25/6, 5) -- (5/1, 6);
            \end{scope}

            \draw[blue!60, line width=2pt] (0,0) -- (0,1) -- (0,2) -- (1,2)node[midway,above]{\textcolor{black}{$2$}} -- (1,3) -- (1,4) -- (2,4)node[midway,above]{\textcolor{black}{$4$}} -- (3,4)node[midway,above]{\textcolor{black}{$1$}} -- (3,5) -- (3,6) -- (4,6) node[midway,above]{\textcolor{black}{$3$}} -- (5,6) node[midway,above]{\textcolor{black}{$0$}};

            \node[below] at (2.5,0) {$\pi = \phi_0(\pi) \colon \dinv = 7$};
        \end{tikzpicture}
    \end{minipage}
    \begin{minipage}{.99\textwidth}
        \begin{tikzpicture}[scale=.6]
            \draw[draw=none, use as bounding box]  (-.5,-.5)rectangle(5.5,6.5);
            \draw[step=1.0, gray!60, thin] (0,0) grid (5,6);

            \draw[blue!60, line width=2pt] (0,0) -- (0,1) -- (0,2) -- (1,2) -- (2,2) -- (2,3) -- (2,4) -- (3,4) -- (3,5) -- (3,6) -- (4,6) -- (5,6);

            \node[below] at (2.5,0) {$\phi_3(\pi) \colon \dinv = 8$};
        \end{tikzpicture}
        \begin{tikzpicture}[scale=.6]
            \draw[draw=none, use as bounding box]  (-.5,-.5)rectangle(5.5,6.5);
            \draw[step=1.0, gray!60, thin] (0,0) grid (5,6);

            \draw[blue!60, line width=2pt] (0,0) -- (0,1) -- (0,2) -- (1,2) -- (2,2) -- (2,3) -- (2,4) -- (3,4) -- (4,4) -- (4,5) -- (4,6) -- (5,6);

            \node[below] at (2.5,0) {$\phi_1(\pi) \colon \dinv = 9$};
        \end{tikzpicture}
        \begin{tikzpicture}[scale=.6]
            \draw[draw=none, use as bounding box]  (-.5,-.5)rectangle(5.5,6.5);
            \draw[step=1.0, gray!60, thin] (0,0) grid (5,6);
            
            \draw[blue!60, line width=2pt] (0,0) -- (1,0) -- (1,1) -- (1,2) -- (2,2) -- (2,3) -- (2,4) -- (3,4) -- (4,4) -- (4,5) -- (4,6) -- (5,6);

            \node[below] at (2.5,0) {$\phi_4(\pi) \colon \dinv = 10$};
        \end{tikzpicture}
        \begin{tikzpicture}[scale=.6]
            \draw[draw=none, use as bounding box]  (-.5,-.5)rectangle(5.5,6.5);
            \draw[step=1.0, gray!60, thin] (0,0) grid (5,6);

            \draw[blue!60, line width=2pt] (0,0) -- (1,0) -- (1,1) -- (1,2) -- (2,2) -- (3,2) -- (3,3) -- (3,4) -- (4,4) -- (4,5) -- (4,6) -- (5,6);

                \node[below] at (2.5,0) {$\phi_2(\pi) \colon \dinv = 11$};
        \end{tikzpicture}
    \end{minipage}
    \caption{A rectangular Dyck path $\pi$ and $\phi_k(\pi)$ for all $0 \leq k < m$. The horizontal steps of $\pi$ are marked by integers indicating their order with respect to the distance between their endpoint and the main diagonal.}
    \label{fig:partition RP}
\end{figure}
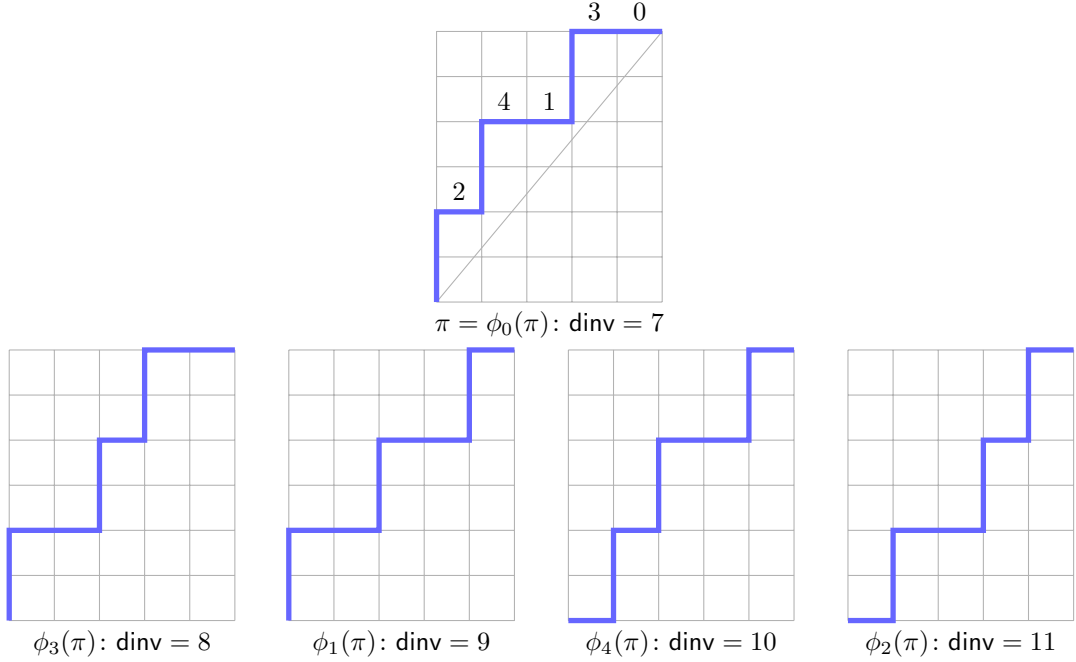
    
\begin{lemma}
    \label{lemma:main_char}
    If $\gcd (m,n)=1$, then $\sweep(\phi_k(\pi))= q^{r_k(\pi)}\sweep(\pi)$.
    \label{lemma}
\end{lemma}

\begin{proof}
    The relative order of points in $\pi$ and their images in $\phi(\pi)$ does not change when performing the sweep process.
    Therefore,
    \[
        \frac{\sweep(\phi(\pi))}{q^{a(\phi(\pi))}}= \frac{\sweep(\pi)}{q^{a(\pi)}},
    \]
    where $a(\pi)$ is the exponent of $q$ obtained by applying the sweep process.
    To conclude, we need to show that $a(\phi(\pi))=a(\pi) + r(\pi)$.
    
    Define $A_\pi, B_\pi, C_\pi, D_\pi$ as the sets of lattice points of $\pi$, different from the point $(m,n)$, that are between a $NE$, $EN$, $NN$, $EE$ pair of steps respectively.
    We consider the point $(0,0)$ to be preceded by a virtual East step, so $(0,0)\in B_\pi$ or $(0,0)\in D_\pi$ if the first step is a North or an East step respectively.
    
    Let $p$ be a lattice point of $\pi$.
    Define $a(p) \in \Z$ as the number of vertical steps that intersect the ray $\rho(p) \coloneqq \{p+u \cdot (m,n) \mid u \in \mathbb{R}_+\}$, multiplied by the following coefficient $\epsilon(p)$:
    \[
        \epsilon(p) =
        \begin{cases}
            0 & \text{if $p \in A_\pi \cup B_\pi$} \\
            -1 & \text{if $p \in C_\pi$} \\
            1 & \text{if $p \in D_\pi$}.
        \end{cases}
    \]
    By construction, we have that $a(\pi) = \sum_{p \in \pi} a_\pi(p)$.
    
    For a lattice point $p$ of $\pi$, denote by $l = l(p) \in \{0, 1\}$ the index such that $p$ is a point of $\pi_l$.
    For this purpose, the right endpoint of the $k$-th horizontal step is considered as a lattice point of $\pi_1$ (not $\pi_0$), whereas $(m,n)$ is not considered as a lattice point of $\pi_1$ (or of $\pi$).
    Define $a'(p) \in \Z$ as the number of vertical steps of $\pi_{1-l}$ that intersect the line $\lambda(p) \coloneqq \{p + u \cdot (m,n) \mid u \in \mathbb{R}\}$, multiplied by the following coefficient $\epsilon'(p)$:
    \[
        \epsilon'(p) =
        \begin{cases}
            0 & \text{if $p \in A_\pi \cup B_\pi$} \\
            (-1)^{l} & \text{if $p \in C_\pi$} \\
            (-1)^{l+1} & \text{if $p \in D_\pi$}.
        \end{cases}
    \]
    In other words: $a'(p)$ vanishes if $p \in A_\pi \cup B_\pi$; otherwise, $|a'(p)|$ is equal to the number of intersections between the line $\lambda(p)$ and vertical steps in the part of the path not containing $p$.
    
    \textbf{Claim 1:} $\displaystyle a(\phi(\pi)) - a(\pi) = \sum_{p \in \pi} a'(p)$.
    
    The intersections between rays $\rho(p) = \{p + u \cdot (m,n) \mid u \in \mathbb{R}_+\}$ and vertical steps in $\pi_{l(p)}$ are counted in both $a(\phi(\pi))$ and $a(\pi)$ (with the same sign), so they simplify.
    
    The remaining summands in $a(\phi(\pi))$ count the intersections between rays $\rho(p)$, where $p$ is in $\pi_1$, and vertical steps in $\pi_0$ (where $\pi_0$ is translated by $(m,n)$ so that it starts from $(m,n)$).
    Equivalently, they count the intersections between lines $\lambda(p)$ (where $p$ is in $\pi_1$) and vertical steps in $\pi_0$ (not translated).
    Therefore, their contribution is given by $\sum_{p \in \pi_1} a'(p)$.
    Note that the points in $C_\pi$ get a negative sign, as in the definition of $a_\pi (p)$.
    
    The remaining summands in $a(\pi)$ count the intersections between rays $\rho(p)$, where $p$ is in $\pi_0$, and vertical steps in $\pi_1$.
    Since $\pi_1$ comes after $\pi_0$, we can substitute the rays $\rho(p)$ with the lines $\lambda(p)$.
    Their contribution is given by $\sum_{p \in \pi_0} a'(p)$.
    
    \textbf{Intermezzo:}
    We refer to a maximal sequence of consecutive North steps as a \emph{vertical segment}. Each point in $D_\pi$ (i.e., between two East steps) is considered as a vertical segment of length $0$.
    This way, the path $\pi_0$ has $k$ vertical segments with $x$ coordinates equal to $0, \dotsc, k-1$, and the path $\pi_1$ has $m-k$ vertical segments with $x$ coordinates $k, \dotsc, m-1$.
    Denote by $S_i$ the $i$-th vertical segment.
    
    It is convenient to translate each vertical segment $S_i$ along the line $\{ u \cdot (m, n) \mid u \in \mathbb{R}\}$ so that its $x$ coordinate becomes $0$.
    We denote this translated segment by $T_i$.
    Let $y_i$ and $y_i'$ be the $y$ coordinates of the endpoints of $T_i$, with $y_i \leq y_i'$.
    Therefore, the $y$ coordinates of $S_i$ are $y_i + i \cdot \frac{n}{m}$ and $y_i' + i \cdot \frac{n}{m}$.
    Note that the endpoints of the $T_i$'s are all distinct because $m$ and $n$ are coprime.
    
    \textbf{Claim 2:} $\displaystyle \sum_{p \in \pi} a'(p) = \sum_{i < k} \sum_{j \geq k} ( \delta_{T_i \supset T_j} - \delta_{T_i \subset T_j} )$.
    
    Let us analyze the contributions to the left hand side due to the $i$-th and $j$-th vertical segments, for fixed $i < k$ and $j \geq k$.
    Let $h$ be the number of integral points $p \in S_i$
    (including the endpoints of $S_i$)
    such that $\lambda(p)$ intersects the $j$-th vertical segment $S_j$.
    
    If $T_i \supset T_j$, then $S_j$ has length $h$ and, for all its $h+1$ points $p'$, the line $\lambda(p')$ intersects $S_i$.
    Once we exclude the endpoints, $h-1$ points remain.
    On the other hand, the endpoints of $S_i$ are not among the $h$ points $p \in S_i$ such that $\lambda(p)$ intersects $S_j$.
    The overall contribution of $S_i$ and $S_j$ to the left hand side is $h-(h-1) = +1$.

    Note that if $T_i \supset T_j$ and $h=0$, then $S_j$ is a single point $p' \in D_\pi$ such that $\lambda(p')$ intersects $S_i$, so it contributes to the left hand side as $+1$.
    In other words, vertical segments of length $0$ can still be regarded as having $h-1 = -1$ integral points other than the endpoints.
    
    Similarly, if $T_i \subset T_j$, then the contribution is $-1$.
    Finally, if neither of $T_i$ and $T_j$ contains the other, $S_j$ also has $h$ points $p'$ such that $\lambda(p')$ intersects $S_i$, so the contribution is $0$.

    \textbf{Claim 3:} $\delta_{T_i \supset T_j} - \delta_{T_i \subset T_j} = \delta_{y_i < y_j}  - \delta_{y_{i+1} < y_{j+1}}$ (where we set $y_m = 0$).
    
    Clearly, we have $\delta_{T_i \supset T_j} - \delta_{T_i \subset T_j} = \delta_{y_i < y_j}  - \delta_{y_i' < y_j'}$.
    The top endpoint of $S_i$ has the same $y$ coordinate as the bottom endpoint of $S_{i+1}$, so $y_i' = y_{i+1} + \frac nm$.
    Similarly, $y_j' = y_{j+1} + \frac nm$, so $\delta_{y_i' < y_j'} = \delta_{y_{i+1} < y_{j+1}}$.
    
    \textbf{Claim 4:} $\displaystyle\sum_{i < k} \sum_{j \geq k} ( \delta_{y_i < y_j}  - \delta_{y_{i+1} < y_{j+1}} ) = r(\pi)$.
    
    Write $\delta_{i, j}$ as a shorthand for $\delta_{y_i < y_j}$.
    The left hand side simplifies to
    \begin{equation}
        \sum_{k \leq j < m} \delta_{0, j} + \sum_{0 < i < k} \delta_{i, k} - \sum_{k < j \leq m} \delta_{k, j} - \sum_{0< i < k} \delta_{i, m}
        = 1 + \sum_{0 \leq i < m} (\delta_{0, i} + \delta_{i, k} - 1),
        \label{eq:sums}
    \end{equation}
    where we have used the facts that $y_m = y_0 = 0$ and $\delta_{i, j} = 1 - \delta_{j,i}$ for $i \neq j$.
    
    If $y_k >0$, the final summation in \eqref{eq:sums} counts the horizontal steps of $\pi$ whose right endpoint lies strictly between the main diagonal and the translated diagonal $\{y_k + u \cdot (m,n) \mid u \in \mathbb{R}\}$.
    The $+1$ term can be interpreted as counting the final horizontal step which ends on the main diagonal.
    
    If $y_k < 0$, the final summation in \eqref{eq:sums} counts the same points with negative sign, but also has a $-2$ coming from the terms $i=0$ and $i=k$ (because $\delta_{0,k} = 0$).
    Then $-2+1 = -1$ counts the final horizontal step with negative sign.
    In all cases, the result is exactly $r(\pi)$.
\end{proof}

This completes the proof of \Cref{thm:coprime-case}.

\section*{Acknowledgements}

The first author would like to thank François Bergeron for suggesting to look into the combinatorics of $p_{m,n}$ and for the discussions on the topic.

The second author is partially supported by H2020 MSCA RISE project GHAIA -- n. 777822, and by National Science Foundation under Grant No. DMS-1929284.

\bibliographystyle{amsalpha}
\bibliography{references}

\end{document}